\documentclass[12pt]{amsart}
\usepackage{amsmath,amssymb,enumerate}
\usepackage{latexsym}
\usepackage{amsfonts}
\usepackage {amsbsy}
\usepackage {amsmath}
\usepackage{amssymb}
\usepackage{enumerate}
\usepackage{ bbm}
\newtheorem{theorem}{Theorem}
\newtheorem{lemma}[theorem]{Lemma}
\newtheorem{corollary}[theorem]{Corollary}
\newtheorem{proposition}[theorem]{Proposition}

\theoremstyle{definition}
\newtheorem{definition}[theorem]{Definition}
\newtheorem{example}[theorem]{Example}
\newtheorem{remark}[theorem]{Remark}

\numberwithin{theorem}{section}
\numberwithin{equation}{section}

\newcommand{\B}{\mathbb{B}}
\newcommand{\N}{\mathbb{N}}
\newcommand{\R}{\mathbb{R}}
\newcommand{\C}{\mathbb{C}}

\makeatletter
\@namedef{subjclassname@2010}{%
  \textup{2010} Mathematics Subject Classification}
\makeatother

\author{Amel Benali}
\address{Amel Benali, University of Gabes, Faculty of Sciences of Gabes,  LR17ES11,  Mathematics and Applications, 6072, Gabes, Tunisia}
\email{amelmath.kabs@gmail.com}

\smallskip

\author{Ahmed Zeriahi}
\address{Ahmed Zeriahi, Institut de Math\'ematiques de Toulouse, 
Universit\'e de Toulouse, 
CNRS, UPS, 118 route de Narbonne, 
31062 Toulouse cedex 09, France}
\email{ahmed.zeriahi@math.univ-toulouse.fr}

\begin{document}

\thanks{The second author was partially supported by the ANR project GRACK}



\title[The H\"older continuous subsolution theorem]{The H\"older continuous subsolution theorem for complex Hessian equations }

\date{\today}

%\contents

\begin{abstract}
 Let  $\Omega \Subset \mathbb C^n$ be a bounded strongly $m$-pseudoconvex domain ($1\leq m\leq n$) and  $\mu$ a positive Borel measure with finite mass on $\Omega$.
Then we solve the H\"older continuous subsolution problem for the complex Hessian equation $(dd^c u)^m \wedge \beta^{n - m} = \mu$ on $\Omega$. Namely, we show that this equation  admits a unique H\"older continuous solution on $\Omega$ with a given H\"older continuous boundary values if it admits a H\"older continuous subsolution on $\Omega$. The main step in solving the problem is to establish a new capacity estimate showing that  the $m$-Hessian measure of a H\"older continuous $m$-subharmonic function on $\Omega$ with zero boundary values is dominated by  the $m$-Hessian capacity with respect to $\Omega$ with an (explicit) exponent $\tau > 1$. 
\end{abstract}

\maketitle

\tableofcontents

\section{Introduction}

Complex Hessian equations are important examples of fully non-linear PDE's of second order on complex manifolds.  They interpolate between (linear) complex Poisson equations ($m = 1$) and (non linear) complex Monge-Amp\`ere equations $(m=n$).
 They appear in many geometric problems, including the $J$-flow
\cite{SW} and quaternionic geometry \cite{AV}. They have attracted the attention of many researchers these last years as we will mention below.

 \subsection{Statement of the problem}
 
Let $\Omega \Subset \C^n$ be a bounded domain and $1 \leq m \leq n$ a fixed integer. 
We consider the following general Dirichlet problem for the complex $m$-Hessian equation : 

\smallskip
\smallskip

{\it The Dirichlet problem:} Let  $g \in \mathcal{C}^{0} (\partial \Omega)$ be a continuous function (the boundary data) and  $\mu$ be a  positive Borel measure on $\Omega$ (the right hand side). The problem is to find a necessary and sufficient condition on $\mu$ such that the following problem admits a solution :   

\begin{equation}\label{eq:DirPb}
\left\{\begin{array}{lcl} 
U \in \mathcal{SH}_m (\Omega) \cap \mathcal {C}^{0} ({\Omega}) \\
 (dd^c U)^m \wedge \beta^{n - m} = \mu   &\hbox{on}\  \Omega  \, \, \, \,  \, \, (\dag)\\
  U_{\mid  \partial \Omega} = g & \hbox{on}\  \partial \Omega \, \, \, \, (\dag \dag)
\end{array}\right.
\end{equation}

The  equation $(\dag)$ must be understood in the sense  of currents  on $\Omega$ as it will be explained in section $2$.
The equality $(\dag \dag)$ means that $\lim_{z \to \zeta} U (z) = g (\zeta)$ for any point $\zeta \in \partial \Omega$. 

Recall that for  a real function $u \in \mathcal{C}^2 (\Omega)$ and each integer $1 \leq k \leq n$,  we denote  by $\sigma_k (u)$ the continuous function defined 	at each point $z \in  \Omega$ as the $k$-th symmetric polynomial of the eigenvalues $\lambda (z) := (\lambda_1 (z), \cdots \lambda_n(z))$ 
of the complex Hessian matrix $  \left(\frac{\partial^2 u }{\partial z_j \partial \bar{z}_k} (z)\right)$ of $u$ i.e. 
$$
\sigma_k (u) (z) := \sum_{1 \leq j_1 < \cdots < j_k \leq n} \lambda_{j_1} (z) \cdots \lambda_{j_k} (z), \, \,  \, \, z \in \Omega.
$$

We say that a real function $u \in \mathcal{C}^2 (\Omega)$ is $m$-subharmonic on $\Omega$ if for any $1 \leq k \leq m$, we have $\sigma_k (u) \geq 0$ pointwise  on $\Omega$.

For $m= 1$, $\sigma_1 (u) = (1 \slash 4) \Delta u$ and for $m = n$, $\sigma_n (u)  = \mathrm{det}  \left(\frac{\partial^2 u }{\partial z_j \partial \bar{z}_k} (z)\right).$
Therefore $1$-subharmonic means subharmonic  and  $n$-subharmonic means plurisubharmonic.

As observed by Z. B\l ocki (\cite{Bl05}), it is possible to define a general notion of $m$-subharmonic functions using the theory of $m$-positive currents (see section 2). Moreover it is possible to define the $k$-Hessian measure $(dd^c u)^k \wedge \beta^{n - k}$ when $1 \leq k \leq m$ for any (locally) bounded $m$-subharmonic function $u$ on $\Omega$ (see section 2). 

When $\mu = 0$, the Dirichlet problem (\ref{eq:DirPb}) can be solved using the Perron method as for the complex Monge-Amp\`ere equation (see \cite{Bl05}, \cite{Ch16a}).

When $g = 0$ and $\mu$ is a positive Borel measure on $\Omega$, the Dirichlet problem is much more difficult. A necessary condition for the existence of a solution to (\ref{eq:DirPb}) is  the existence of a subsolution.  

Therefore a particular case of the Dirichlet problem (\ref{eq:DirPb}) we are interested in can be formulated as follows.    

\smallskip

{\it The H\"older continuous subsolution problem :}
 Let $\mu$ be a positive Borel measure on $\Omega$. Assume that there exists a function $\varphi \in  \mathcal{SH}_m (\Omega) \cap \mathcal{C}^{\alpha} (\bar{\Omega})$ satisfying the following condition :
\begin{equation} \label{eq:subsolution}
\mu \leq (dd^c \varphi)^m \wedge \beta^{n - m}, \, \, \mathrm{on} \, \, \, \, \Omega,  \, \, \, \mathrm{and} \, \, \varphi_{\mid  \partial \Omega} = 0.
\end{equation}

1.  Does the  Dirichlet problem (\ref{eq:DirPb}) admit a H\"older continuous  solution $U_{\mu,g}$ for any  boundary data $g$ which is H\"older continuous on  $\partial{\Omega}$? 
 
2.  In this case, is it possible to estimate  precisely the H\"older exponent of the solution $U_{\mu,g}$ in terms of the  H\"older exponents of  $\varphi$ and  $g$ ?

\smallskip
Our goal in this paper is to answer the first question on the existence of a H\"older continuous solution and give an explicit lower bound of  the H\"older exponent of the solution in terms of the H\"older exponent of the subsolution when the measure $\mu$ has finite total mass.

\subsection{Known results}
There have been many articles on the subject. We will only mention those that are relevant to our study and closely related to our work.
The terminology used below will be defined in the next section.

Assume that $\Omega$ is a smooth strongly $m$-pseudoconvex domain. When the boundary data $g$ is  smooth and the right hand side $\mu = f \lambda_{2n}$ is a measure  with  a smooth positive density $f > 0$, S.Y.  Li proved in \cite{Li04} that the problem has a unique smooth solution. Later, Z. B\l ocki introduced the notion of weak solution and solved the Dirichlet problem for the homogenous Hessian equation in the unit ball in $\C^n$ (\cite{Bl05}).
When the density $0 \leq f \in L^p (\Omega)$ with $p > n \slash m$, Dinew and Ko\l odziej proved the existence of a continuous solution (\cite{DK14}). Assuming moreover that $g$ is H\"older continuous on $\bar \Omega$,  Ngoc Cuong Nguyen proved the H\"older continuity of the solution  under an additional assumption on the density $f$ (\cite{N14}). The general case was considered in \cite{BKPZ16} and (\cite{Ch16}).

\smallskip

On the other hand, S. Ko\l odziej \cite{Kol05} proved that the Dirichlet problem has a bounded plurisubharmonic  solution if (and only if) it has a bounded subsolution with zero boundary values. This is known as the bounded subsolution theorem for plurisubharmonic functions. The same result was proved for the Hessian equation by Ngoc Cuong Nguyen in \cite{N13}.

The H\"older continuous subsolution problem stated above has attracted a lot of attention these last years and was formulated  in \cite{DGZ16} for the complex Monge-Amp\`ere equation.

\smallskip
It has been solved  for the complex Monge-Amp\`ere by Ngoc Cuong Nguyen in \cite{N18a,N18b}.
Recently S. Kolodziej and Ngoc Cuong  Nguyen  solved the H\"older subsolution problem for the Hessian equation under the restrictive assumption that the measure $\mu$ is compactly supported on $\Omega$ (see \cite{KN18}, \cite{KN19}).

\subsection{Main new results}

In this paper we will solve the H\"older continuous subsolution problem for Hessian equations when  $\mu$ is any positive Borel measure with finite mass on $\Omega$. 

Our first main result gives a new comparison inequality which will be applied to  positive Borel measures without restriction on their support. 

 \smallskip
 \smallskip
 
 {\bf Theorem A}.{ \it Let $\Omega \Subset \C^n$ be a  bounded strongly $m$-pseudoconvex  domain.
 Let $\varphi\in \mathcal{SH}_m(\Omega)\cap \mathcal{C}^{\alpha}(\overline\Omega)$ with $0 < \alpha \leq 1$ such that $\varphi = 0$ in $\partial \Omega$.  Then for any $0 < r < m \slash (n-m)$, there exists a constant $A>0$ such that for every compact $K\subset\Omega$,
 $$
\int_ K (dd^c\varphi)^m\wedge\beta^{n-m} \leq A \, \left(\left[\mathrm{Cap}_m (K,\Omega)\right]^{1 + \epsilon} + \left[\mathrm{Cap}_m (K,\Omega)\right]^{1 + m\epsilon}\right), 
$$
 where $\epsilon := \frac{\alpha r}{(2-\alpha) m + \alpha} > 0$.
}

 \smallskip
 \smallskip
 
  The capacity $\mathrm{Cap}_m (K,\Omega)$ will be defined in the next section. The constant $A$ in the theorem is explicit (see formula (\ref{eq:finalConst})).
 
 Observe that the most relevant case in the application of this inequality will be when $\mathrm{Cap}_m (K,\Omega) \leq 1$. In this case the right exponent is $\tau = 1 + \frac{\alpha r}{(2-\alpha) m + \alpha}$.

 Theorem A improves substantially a recent result of \cite{KN19} who proved an  estimate of this kind when the compact set $K \subset  \Omega'$ is contained in a fixed open set $\Omega' \Subset \Omega$, i.e. $K$ stays away from the boundary of $\Omega$.

 When $m=n$ a better estimate was obtained in \cite{N18a} using the exponential integrability of plurisubharmonic functions which fails when $m < n$. 

\smallskip
 
 As a consequence of Theorem A, we will deduce the following result which solves the H\"older continuous subsolution problem.

 \smallskip
 
 \smallskip

 {\bf Theorem B}. { \it Let $\Omega \Subset \C^n$ be a  bounded  strongly  $m$-pseudoconvex  domain and $\mu $ a positive Borel measure on $\Omega$ with finite mass. Assume that there exists $\varphi\in \mathcal{E}^0_m(\Omega)\cap\mathcal{C}^{\alpha}(\overline\Omega)$ with $0 < \alpha \leq  1$ such that   
\begin{equation} \label{eq:subsol}
 \mu \leq (dd^c\varphi)^m\wedge\beta^{n-m}, \, \, \, \mathrm{weakly \, \, on} \, \,  \Omega, \, \, \mathrm{and} \, \, \, \varphi\mid_{\partial \Omega} \equiv 0.
\end{equation} 

Then for any   boundary datum $g $ H\"older continuous on $\partial \Omega$,   the Dirichlet problem (\ref{eq:DirPb}) admits a unique solution $U = U_{g,\mu} $ which is H\"older continuous on $\bar{\Omega}$.
 More precisely, 

1)  if  $g \in  \mathcal C^{1,1} (\partial \Omega)$, $ U \in \mathcal{C}^{\alpha'}(\overline\Omega)$ for any  $0 < \alpha' <  2 \gamma (m,n,\alpha) \frac{\alpha^m}{2^{m}}$, 
where
 \begin{equation} \label{eq:gamma'}
 \gamma (m,n,\alpha) := \frac{ m \alpha}{ m (m + 1) \alpha  +  (n-m) [(2 - \alpha)m + \alpha]},
 \end{equation} 
 
2)  if $g \in  \mathcal C^{2 \alpha} (\partial \Omega)$, then $ U \in \mathcal{C}^{\alpha''}(\overline\Omega)$ for any  $0 < \alpha'' <  \gamma' (m,n,\alpha) \frac{\alpha^m}{2^{m}}$,  where
$$
 \gamma' (m,n,\alpha) := \frac{\alpha}{ m (m + 1) \alpha  +  (n-m) [(2 - \alpha)m + \alpha]}\cdot
$$ } 

\smallskip
 \smallskip
 
 Recall that by definition when $\alpha = 1 \slash 2$, $g \in \mathcal C^1 (\partial \Omega)$ means that $g$ is Lipschitz and when $1 \slash 2 < \alpha \leq1$ and  $2 \alpha = 1 + \theta$ with $0 < \theta \leq 1$,  $ g \in \mathcal{C}^{2 \alpha} (\partial \Omega)$ means that $g \in  \mathcal{C}^1 (\partial \Omega)$ and  and $\nabla g$ is H\"older continuous of exponent $\theta$ on $\partial \Omega$.

 \smallskip
 Let us give a rough idea of the proofs of these results.
 
 \smallskip

  {\it Idea of the proof of Theorem A:} The general idea of the proof is inspired by \cite{KN19}. However, since our measure is not compactly supported nor of finite mass, we need to control the behaviour of the $m$-Hessian measure of $\varphi$ close to the boundary.    This will be done in several steps in section 3 and section 4. 
 
 - The first step is to estimate the mass of the $m$-Hessian measure $\sigma_m (\varphi)$ of a H\"older continuous $m$-subharmonic function $\varphi$ in terms of its regularization $\varphi_\delta$ on any compact set in $\Omega_\delta$. This requires to consider the $m$-subharmonic envelope of $\varphi_\delta$  on $\Omega$ and provide a precise control on its $m$-Hessian measure (see Theorem \ref{thm:obstacle}).
 
 -  The second step is to estimate the mass  of $\sigma_m (\varphi)$  on a compact set close to the boundary in terms of its Hausdorff distance to the boundary  (see Lemma \ref{lem:ComparisonIneq}).   
 
 \smallskip
 \smallskip
 
 {\it  Idea of the proof of Theorem B:} The proof   will be in two steps. 
 
 - The first step relies on a standard method which goes back to \cite{EGZ09} (see also \cite{GKZ08}) in the case of the complex Monge-Amp\`ere equation. This method  consists in proving a  semi-stability inequality  estimating  
$ \sup_{\Omega} (v-u)_+ $ in terms of  $\Vert (v-u)_+\Vert_{L^{1} (\Omega,\mu)}$, where $u $ is the bounded $m$-subharmonic solution to the Dirichlet problem (\ref{eq:DirPb})  and $v$ is any bounded $m$-subharmonic function with the same boundary values as $u$, under the assumption that the measure $\mu$ is dominated by the  $m$-Hessian capacity with an exponent $\tau > 1$ (see Definition \ref{def:cap-domination}).
 
 - The second step  uses an idea which goes back to \cite{DDGKPZ15} in the setting of compact K\"ahler manifolds (see also \cite{GZ17}). It has been also used in the local setting in  \cite{N18a} and \cite{KN19}. It consists in estimating the $L^1 (\mu)$-norm of $v - u$  in terms of the $L^1 (\lambda_{2 n})$-norm of $(v-u)$ where $u$ is the bounded solution to the Dirichlet problem and $v$ is a bounded $m$-subharmonic function on $\Omega$ close to the regularization $ u_\delta$ of $u$. This step requires that the measure $\mu$ is well dominated by  the $m$-Hessian capacity, which is precisely the content of our Theorem A.
  Then using the Poisson-Jensen formula as in \cite{GKZ08}, we see that the $L^1$-norm of $(u_\delta-u)$ is $O (\delta)$ (see Lemma \ref{lem:Poisson-Jensen}) and Lemma \ref{lem:sup-mean}  allows us to finish the proof.

 \section{Preliminary results}
 In this section, we recall the basic properties of $m-$subharmonic functions and some results we will use  throughout the paper. 
 
 \subsection{Hessian potentials}
 For a hermitian $n \times n$ matrix $a = (a_{j,\bar k})$ with complex coefficients, we denote by $\lambda_1, \cdots \lambda_n$ the eigenvalues of the matrix $a$. For any $1 \leq k \leq n$ we define the $k$-th trace of $a$ by the formula

$$
s_k (a) := \sum_{1 \leq j_1 < \cdots < j_k \leq n} \lambda_{j_1} \cdots \lambda_{j_k},
$$
which is the $k$-th elementary symetric polynomial of the eigenvalues $(\lambda_1, \cdots, \lambda_n)$ of $a$.

 Let $\C^n_{(1,1)} $ be the space of real $(1, 1)$-forms on $\C^n$  with constant
coefficients, and define the cone of $m$-postive  $(1,1)$-forms on $\C^n$ by

\begin{equation}\label{eq:mpositive}
\Theta_m := \{\theta \in \C^n_{(1,1)}  \, ; \,  \theta \wedge  \beta^{n - 1} \geq 0, \cdots,  \theta^m \wedge  \beta^{n - m} \geq 0\}.
\end{equation}

\begin{definition} \label{def:mpositive}
1) A smooth $(1,1)$-form $\theta$ on $\Omega$ is said to be $m$-postive on $\Omega$ if for any $z \in \Omega$, $\theta (z) \in \Theta_m$.

2) A function $u:\Omega \rightarrow \mathbb{R}\cup\{-\infty\}$ is said to be  $m-$subharmonic  on $\Omega$ if it is subharmonic on $\Omega$ (not identically $-\infty$ on any component) and  for any collection of smooth $m-$positive $(1,1)-$forms  $\theta_1,...,\theta_{m-1}$ on $\Omega$, the following inequality 
  $$
  dd^c u\wedge \theta_1\wedge...\theta_{m-1} \wedge \beta^{n-m}\geq 0,
  $$
 holds in the sense of currents on $\Omega$.
\end{definition}

We denote by $\mathcal{SH}_m (\Omega) $ the positive convex cone of $m$-subharmonic functions on $\Omega$.

We give below the most basic properties of $m$-subharmonic functions that will be used in the sequel.

\begin{proposition}\label{prop:basic}

\noindent 1.  If $u\in \mathcal{C}^2(\Omega)$, then $u$ is  $m$-subharmonic on $\Omega$ if and only if $(dd^c u)^k\wedge \beta^{n-k}\geq0$
pointwise on $\Omega$ for $k=1, \cdots, m$.

 \noindent 2. $\mathcal{PSH}(\Omega)=\mathcal{SH}_n(\Omega)\subsetneq \mathcal{SH}_{n-1}(\Omega)\subsetneq...\subsetneq \mathcal{SH}_1(\Omega)=\mathcal{SH}(\Omega) $.
 
\noindent 3.  $\mathcal{SH}_m(\Omega) \subset L^1_{loc} (\Omega)$ is a positive convex cone. 
  
\noindent 4.  If $u$ is $m$-subharmonic on $\Omega$ and $f: I \rightarrow\mathbb{R}$ is a  convex, increasing function on some interval containing the image of $u$, then $f\circ u$ is $m$-subharmonic on $\Omega$.

\noindent 5. The limit of a decreasing sequence of  functions in $\mathcal{SH}_m(\Omega)$ is $m$-subharmonic on $\Omega$ when it is not identically $- \infty$ on any component.
 
\noindent 6.  Let $u \in  \mathcal{SH}_m(\Omega)$ and $v \in \mathcal{SH}_m(\Omega') $, where $\Omega'\subset\C^n$ is an open set such that $\Omega \cap \Omega' \neq \emptyset$. If  $u\geq v$ on $\Omega \cap \partial\Omega'$, then the function
   $$
   z \mapsto w(z):=\left\{\begin{array}{lcl}
\max(u(z),v(z)) &\hbox{ if}\ z \in \Omega \cap \Omega'\\
 u(z)  &\hbox{if}\  z \in\Omega\setminus\Omega'\\
\end{array}\right.
$$
is $m$-subharmonic on $\Omega$.

 \end{proposition}

 Another ingredient which will be important is the regularization process.  Let $\chi$ be a fixed smooth positive radial function with compact support in the unit ball $\B \subset \C^n$ and $\int_{\mathbb{C}^n}\chi (\zeta)d\lambda_{2 n}(\zeta)=1$.
   For any $ 0 < \delta < \delta_0 := \mathrm{diam} (\Omega)$, we set 
   $\chi_{\delta}(\zeta)=\frac{1}{\delta^{2n}}\chi (\frac{\zeta}{\delta})$ and $\Omega_{\delta}=\{z \in\Omega;  \mathrm{dist} (z,\partial\Omega)>\delta\}$.
 
 Let $u \in \mathcal{SH}_m (\Omega) \subset L^1_{loc} (\Omega)$ and define its standard $\delta$-regularization by the formula
      
  \begin{equation} \label{eq:reg}
  { u}_{\delta} (z) := \int_{\Omega} u (z - \zeta) \chi_{\delta} (\zeta) d \lambda_{2n} (\zeta), z \in \Omega_{\delta}.
  \end{equation}
  Then it is easy to see that $ {u}_{\delta}$ is $m$-subharmonic and smooth on $\Omega_{\delta}$ and decreases to $u$ on $\Omega$ as $\delta $ decreases to $0$.
  
  The following lemma was proved in \cite{GKZ08}  (see  also \cite{Ze20}).

\begin{lemma}   \label{lem:Poisson-Jensen}
Let $u \in \mathcal{SH}_m (\Omega) \cap L^{1} (\Omega)$. Then for $0 < \delta < \delta_0$, its $\delta$-regularization extends to $\C^n$ by the formula
      
  \begin{equation} \label{eq:regularization}
  { u}_{\delta} (z) := \int_{\Omega} u (\zeta) \chi_{\delta} (z - \zeta) d \lambda_{2n} (\zeta), z \in \C^n,
  \end{equation}
and have the following properties :
 
\noindent 1)  $ { u}_{\delta}$ is a smooth function on $\C^n$ which is $m$-subharmonic on $\Omega_\delta$ and $({u}_\delta)$ decreases to $u$ on $\Omega$ as $\delta $ decreases to $0$;
      
\noindent 2)  for any $0 < \delta < \delta_0$, we have
  \begin{equation} \label{eq:PJ1}
  \int_{\Omega_\delta} ( { u}_{\delta} (z) - u(z)) d\lambda_{2 n}(z) \leq a_n \delta^2 \int_{\Omega_\delta} dd^c u \wedge \beta^{n - 1},
  \end{equation} 
  where $a_n > 0$ is a uniform constant independant of $u$ and $\delta$.
  
  3)  there exists a constant $b_n > 0$ such that if $u \leq 0$ on $\Omega$,
   \begin{equation}  \label{eq:PJ2}
   \int_{\Omega_\delta} (u_\delta - u) d \lambda_{2n}   \leq b_n \delta \Vert u\Vert_1,
\end{equation}   
where $\Vert u\Vert_1 := \int_\Omega \vert u\vert d \lambda_{2 n}$.
\end{lemma}
\begin{proof} The first property is trivial and the second is proved in \cite{GKZ08}. The thid property follows easily from the second. 

Indeed, since the defining function  $\rho$ of $\Omega$ is smooth and $\vert \nabla \rho \vert > 0$ on $\partial \Omega$, it follows that  there exists a uniform constant $ c_1 > 0$ such that $- \rho (z) \geq  c_1 \, \text{dist} (z, \partial \Omega)$ (see \cite{Ze20} for more details). Then by the integration by parts inequality (\ref{eq:testinequality}), it follows that 
 \begin{eqnarray*} 
 \int_{\Omega_\delta} dd^c u \wedge \beta^{n-1}  & \leq &  c_2  \delta^{-1} \int_{\Omega} (- \rho)  dd^c u \wedge \beta^{n-1}   \\
 &\leq &   c_3  \delta^{-1} \int_{\Omega} (-u) \beta^n, \nonumber
  \end{eqnarray*}
 where $ c_2,   c_3 > 0$ are uniform constants. The inequality (\ref{eq:PJ2}) follows using (\ref{eq:PJ1}.
 \end{proof}
An estimate like  (\ref{eq:PJ2}) was first obtained in  \cite{BKPZ16} (see also \cite{KN19} and\cite{Ze20})).

Let us introduce  the notion of strong $m$-pseudoconvexity that will be used in the sequel.

\begin{definition}  We say that the open set $\Omega$ is  strongly $m$-pseudoconvex if $\Omega$ admits a  defining function $\rho$ which is smooth strictly $m$-subharmonic in a neighbourhood of $\bar \Omega$ and  $\vert \nabla \rho\vert > 0$ on $\partial \Omega = \{\rho = 0\}$. In this case we can choose $\rho $ so that
\begin{equation} \label{eq:stronpconvexity}
(dd^c \rho)^k \wedge \beta^{n - k} \geq \beta^n \, \, \mathrm{for} \, \, 1 \leq k \leq m,
\end{equation}
pointwise on $\Omega$.
\end{definition}

The following lemma is analoguous to a lemma proved in \cite{GKZ08} using mean values rather than convolution.
\begin{lemma} \label{lem:sup-mean} Let $\Omega \Subset \C^n$ be a bounded domain and $u \in  \mathcal{SH} (\Omega) \cap L^{\infty} ({\bar \Omega})$. Assume that $u$ is H\" older continuous  near $\partial\Omega$ with  exponent $\alpha \in ]0,1[$. Then the following properties are equivalent:

$(i)$ $\exists c_1 >0$, ${ u}_\delta := u \star \chi_\delta \leq u  +  c_1 \delta^{\alpha}$ in $\Omega_\delta$,

$(ii)$ $\exists c_2 >0$, $\sup_{\bar B(z,\delta)} u \leq  u +  c_2 \delta^{\alpha}$ in $\Omega_\delta$,

$(iii)$ $\exists c_3 > 0$, $\sup \vert u(z) - u(z')\vert \leq c_3 \vert z-z'\vert^\alpha$, for $z, z' \in \Omega$.

\end{lemma}
 A  similar lemma  has been recently proved in the compact Hermitian manifold setting in \cite{LPT20}. 
A slight modification of the proof of \cite{GKZ08} with an observation from \cite{LPT20}  works also in our context as it is explained in \cite{Ze20}.

 \begin{remark} \label{rem:HolderBoundary} Recall that $u$ is H\"older continuous  near $\partial\Omega$ with exponent $\alpha \in ]0,1]$ if there exists $\delta_1 > 0$ small enough and a constant $\kappa > 0$ such that for any $\zeta \in \partial \Omega$ and any $0 < \delta < \delta_1$,  
 $$
 \sup_{z \in \Omega (\zeta,\delta)} \vert u (z) - u (\zeta) \vert \leq \kappa \delta^\alpha, \, \, \, \hbox{where} \, \, \, \, \Omega(\zeta,\delta) := \Omega \cap B (\zeta,\delta).
 $$ 

Assume that there exists two functions $v , w $ defined and H\"older continuous with exponent $\alpha $ on a neighbourhood $ U$ of $\partial \Omega$ in $\bar \Omega$ such that 
$v \leq u \leq w$ on $ U$ and $v = u = w$ on $\partial \Omega$. Then $u$ is H\"older continuous with exponent $\alpha $ near $\partial \Omega$.
\end{remark}

\subsection{Complex Hessian operators}

Following \cite{Bl05}, we can define the Hessian operators acting on (locally) bounded $m$-subharmonic functions as follows. 
%Observe that if $u \in \mathcal{SH} (\Omega)$ then $dd^u \wedge \beta^{n-m}$ is a well defined $(m-1,(m-1)$-positive current on $\Omega$.  
 Given $u_1, \cdots, u_k \in \mathcal{SH}_m (\Omega) \cap L^{\infty} (\Omega)$ ($1 \leq k \leq m$), one can define inductively the following  positive $(m-k,m-k)$-current on $\Omega$
$$
dd^c u_1 \wedge \cdots \wedge dd^c u_k \wedge \beta^{n - m} := dd^c (u_1 dd^c u_2 \wedge \cdots \wedge dd^c u_k \wedge \beta^{n - m}).
$$
%When $k = m$ we obtain a positive $(n,n)$-current on $\Omega$ which can be identified to a Borel measure on $\Omega$.

In particular, if $u  \in \mathcal{SH}_m (\Omega) \cap L^{\infty}_{loc} (\Omega)$, the positive current $(dd^c u)^m \wedge \beta^{n-m}$ can be identifed to a  positive Borel measure on $\Omega$, the so called $m$-Hessian measure of $u$ denoted by:
$$
\sigma_m (u) := (dd^c u)^m \wedge \beta^{n-m}.
$$

Observe that when $m= 1$,  $\sigma_1 (u) = dd^c u \wedge \beta^{n-1}$ is the Riesz measure of $u$ (up to a positive constant), while  $\sigma_n (u) = (dd^c u)^n$ is the complex   Monge-Amp\`ere measure of $u$. 

It is then possible to extend Bedford-Taylor theory to this context. 
In particular, Chern-Levine Nirenberg inequalities holds and the Hessian operators are continuous under local uniform convergence and pointwise a.e. monotone convergence 
on $\Omega$ of sequences of functions in  $\mathcal{SH} (\Omega) \cap L^{\infty}_{loc} (\Omega)$ (see \cite{Bl05}, \cite{Lu12}).

We define  $\mathcal{E}_m^0 (\Omega) $ to be the positive convex cone of  negative  functions  $\phi \in \mathcal{SH}^-_m (\Omega) \cap L^{\infty} (\Omega)$ with zero boundary values such that
$$
\int_{\Omega} (dd^c \phi)^m \wedge \beta^{n - m} < + \infty.
$$
These are the "test functions" in  $m$-Hessian Potential Theory integration by parts formula is valid for these functions.

More generally it follows from \cite{Lu12,Lu15} that the following property hlods:  if $\phi \in  \mathcal{E}_m^0 (\Omega) $ and $u , v \in \mathcal{SH}_m (\Omega) \cap L^{\infty} (\Omega)$ with $u \leq 0$, then for $0 \leq k \leq m - 1$, 

\begin{equation} \label{eq:testinequality}
\int_\Omega (-\phi)  dd^c u \wedge (dd^c v)^k \wedge \beta^{n - k-1} \leq \int_\Omega (-u)  dd^c \phi \wedge (dd^c v)^k \wedge \beta^{n - k-1}. 
\end{equation}

An important tool in the corresponding Potential Theory is the Comparison Principle.

\begin{proposition} \label{prop:Comparison Principle}
 Assume that $u,v\in \mathcal{SH}_m(\Omega)\cap L^{\infty}(\Omega)$ and for any $\zeta \in \partial \Omega$, $\liminf_{z \rightarrow \zeta }(u(z)- v(z))\geq 0$.  Then 
 $$
 \int_{\{u<v\}}(dd^c v)^m\wedge\beta^{n-m} \leq \int_{\{u<v\}}(dd^c u)^m\wedge\beta^{n-m}.
 $$
 Consequently, if $(dd^cu)^m\wedge\beta^{n-m}\leq(dd^cv)^m\wedge\beta^{n-m}$ weakly on $\Omega$, then $u \geq v$ on $\Omega$.
 
\end{proposition}
It follows from the comparison principle that if the Dirichlet problem (\ref{eq:DirPb}) admits a solution, then it is unique.

The following result will be  also needed.
\begin{corollary} \label{coro:Comparison Principle} Let $\Omega \Subset \C^n$ be a bounded strongly $m$-pseudoconvex domain. Assume that  $u,v\in \mathcal{SH}_m(\Omega)\cap L^{\infty}(\Omega)$ satisfy $u \leq v$ on $\Omega$ and for any $\zeta \in \partial \Omega$, $\lim_{z \rightarrow \zeta }(u(z)- v(z))= 0$.
Then  for any $\psi \in \mathcal{SH}_m (\Omega) \cap L^{\infty} (\Omega)$ and any $1 \leq k \leq m-1$,
 $$
 \int_{\Omega} dd^c v \wedge (dd^c \psi)^{k} \wedge\beta^{n-k-1} \leq \int_{\Omega} dd^c u \wedge (dd^c \psi)^k \wedge\beta^{n-k-1}.
 $$
\end{corollary}
\begin{proof} Fix $\varepsilon > 0$.  From the hypothesis, the exists a compact subset $K \Subset \Omega$ such that $ u \geq v - \varepsilon$ on $\Omega \setminus K$. Then $v_\varepsilon := \max \{u,v-\varepsilon\} \in \mathcal{SH}_m (\Omega) \cap L^{\infty} (\Omega)$ and $v_\varepsilon = u$ on $\Omega \setminus K$.
We claim that this implies that
\begin{equation} \label{eq:masspreservation}
\int_{\Omega} dd^c v_\varepsilon \wedge (dd^c \psi)^{k} \wedge\beta^{n-k-1} = \int_{\Omega} dd^c u \wedge (dd^c \psi)^k \wedge\beta^{n-k-1}.
\end{equation}
Indeed we have in the sense of currents
$$
dd^c v_\varepsilon \wedge (dd^c \psi)^{k} \wedge\beta^{n-k-1} -  dd^c u \wedge (dd^c \psi)^k \wedge\beta^{n-k-1} = dd^c  T,
$$
where $T :=  (v_\varepsilon - u)  (dd^c \psi)^k  \wedge\beta^{n-k-1}$.

Since $T$ is a current of order $0$ with compact support in $\Omega$, it follows that $\int_\Omega dd^c T = 0$, which proves (\ref{eq:masspreservation}).

Now observe that $v_\varepsilon$ increases to $v$ as $\varepsilon$ decreases to $0$. By the monotone continuity of the Hessian operators, it follows that 
$$  dd^c v_\varepsilon \wedge (dd^c \psi)^{k} \wedge\beta^{n-k-1}  \to  dd^c v \wedge (dd^c \psi)^{k} \wedge\beta^{n-k-1}  $$
 weakly on $\Omega$ as $\varepsilon \to 0$. Therefore using (\ref{eq:masspreservation}) we conclude that
\begin{eqnarray*}
\int_{\Omega} dd^c v \wedge (dd^c \psi)^{k} \wedge\beta^{n-k-1} &\leq  &\liminf_{\varepsilon \to 0} \int_{\Omega} dd^c v_\varepsilon \wedge (dd^c \psi)^{k} \wedge\beta^{n-k-1} \\
&  =&  \int_{\Omega} dd^c u \wedge (dd^c \psi)^{k} \wedge\beta^{n-k-1}.
\end{eqnarray*}
\end{proof}

Let us recall the following estimates due to Cegrell (\cite{Ceg04}) for the complex Monge-Amp\`ere operators and extended by Charabati to  complex Hessian operators (\cite{Ch16}).

\begin{lemma} \label{lem:Cegrell} Let  $u, v, w \in\mathcal{E}_m^0(\Omega)$. Then for any $1 \leq k \leq m - 1$
   
 $$
    \begin{array}{lcl}
 \int_{\Omega}dd^cu\wedge(dd^cv)^k\wedge(dd^cw)^{m-k-1}\wedge\beta^{n-m}
     \leq  H_m (u)^{\frac{1}{m}} \, H_m (v)^{\frac{k}{m}} \,  H_m (w)^{\frac{m-k-1}{m}},
  \end{array}
 $$
 where $H_m (u) :=  \int_{\Omega}(dd^c u)^m \wedge \beta^{n-m}$.
 
 In particular, if $\Omega$ is strongly $m$-pseudoconvex, then
 $$
 \int_{\Omega}dd^c u \wedge (dd^c w)^k \wedge\beta^{n-k -1}  \leq  c_{m,n} \left(I_m (u)\right)^{\frac{1}{m}} \left(I_m (w)\right)^{\frac{k}{m}},
 $$
 and 
$$ 
  \int_{\Omega}dd^c u  \wedge \beta^{n-1}  \leq  c_{m,n} \left(I_m (u)\right)^{\frac{1}{m}},
 $$
 where $c_{m,n} > 0$ is a uniform constant.
\end{lemma}

We will need the following generalization of of last part of Lemma \ref{lem:Cegrell} to functions with boundary values not vanishing identically.
\begin{lemma}  \label{lem:Cegrell2} Assume that $g \in C^{1,1} (\partial \Omega)$. Then there exists a constant  $M' = M' (g,\Omega) > 0$ such that for any $0 \leq k \leq m - 1$, any $v \in \mathcal{SH}_m (\Omega)$ with $v\mid_{\partial \Omega}  \equiv g$ and any $\psi \in \mathcal{E}^0_m (\Omega)$, we have 
 \begin{equation} \label{eq:claim}
 \int_\Omega dd^c v \wedge (dd^c \psi)^{k} \wedge \beta^{n-k - 1}  \leq  \,\left(  c_{m,n}  \,  H_m (v)^{1 \slash m}  + M'\right) \, H_m (\psi)^{k)\slash m} ,
 \end{equation}
 where $ c_{m,n} > 0$ is the same constant as in the previous lemma.
 \end{lemma}
 \begin{proof}  Fix $0 \leq k \leq m-1$ and set 
 $$
 I_k(v,\psi) :=  \int_\Omega dd^c v \wedge (dd^c \psi)^{k} \wedge \beta^{n-k-1}.
 $$ 
 If  $g\mid_{\partial \Omega} \equiv 0$, then $v \equiv 0$ on $\partial \Omega$, and  the statement with $M' = 0$ follows from   Lemma \ref{lem:Cegrell}.
 \smallskip
 
Now assume that $g \in C^{1,1} (\partial \Omega)$.
There exists  $G \in  C^{1,1} (\bar{\Omega})$ such that $G =   g$ on $\partial \Omega$. By the choice of $\rho $ we can find a large constant $L > 0$ such that $w := L \rho + G$ is $m$-subharmonic on $\Omega$ and for $1 \leq k \leq m$, $(dd^c w)^k \wedge \beta^{n-k} \leq L'_m \beta^n$ pointwise almost everywhere on $\Omega$  for some uniform constant $L'_m > 0$.

By the bounded subsolution theorem (\cite{N13}) there exists $v_0 \in \mathcal{E}^0_m (\Omega)$  solution to the equation $(dd^c v_0)^m \wedge \beta^{n-m} = (dd^c v)^m \wedge  \beta^{n-m}$ with  boundary values $v_0 \equiv 0$. 
 The functions $\tilde v := v_0 + w$ is $m$-subharmonic and bounded on $\Omega$ and $ \tilde v  = g = v$ on $\partial \Omega$. By the comparison principle $\tilde v \leq v$ on $\Omega$.
Moreover by Corollary \ref{coro:Comparison Principle} we have
$$
I_k (v,\psi) \leq I_k (\tilde v,\psi).
$$

It suffices to estimate  $ I_k (\tilde v,\psi) $ by a uniform constant. We have
$$
 I_k (\tilde v,\psi)  =   I_k (v_0, \psi) + I_k (w,\psi)
 $$
 
 Since $v_0 \mid_{\partial \Omega} \equiv 0$, from the previous case it follows that
 \begin{eqnarray*}
 I_k (v_0,\psi)  & \leq & c_{m,n} \left(\int_\Omega (dd^c v_0)^m \wedge \beta^{n-m}\right)^{1 \slash m}  \left(\int_\Omega (dd^c  \psi)^m \wedge \beta^{n-m}\right)^{k \slash m}  \\
 &\leq & c_{m,n} H_m (v)^{1 \slash m}  H_m (\psi)^{k\slash m}.
 \end{eqnarray*}
 
It remains to estimate $ I_k (w,\psi)$.  
Since $w \in C^{1,1} (\bar{\Omega})$, it follows that $dd^c w \leq M_3 \beta$ pointwise almost everywhere on $\Omega$, hence by  Lemma  \ref{lem:Cegrell}, we have
 
 \begin{eqnarray*}
\int_\Omega dd^c w \wedge (dd^c \psi)^k \wedge \beta^{n-k-1}  &\leq & M' \int_\Omega (dd^c \psi) ^k \wedge \beta^{n-k} \\
& \leq & M' H_m (\psi)^{k\slash m},
\end{eqnarray*}
 since $R \geq 1$, where $M' = M'(g) > 0$ depends on  the uniform bound of  $dd^c G$.  
 This proves the inequality of the lemma.
 \end{proof}
 
\subsection{The bounded subsolution theorem} 

Let   $\Omega \Subset \C^n$ be a bounded strongly $m$-pseudoconvex domain. 

Assume there exists $v \in \mathcal{SH}_m (\Omega)\cap L^{\infty} (\Omega)$ such that 
\begin{equation} \label{eq:boundedsubsol}
\mu \leq (dd^c v)^m \wedge \beta^{n -m} \, \, \mathrm{on} \, \, \Omega \, \, \, \hbox{and} \,  \, \, v|_{\partial \Omega} \equiv 0.
\end{equation}
 Ngoc Cuong  Nguyen proved that under this condition, the Dirichlet problem (\ref{eq:DirPb}) admits a unique bounded $m$-subharmonic solution (see \cite{N13}). 

\begin{theorem} \label{thm:boundedsubsolution} (\cite{N13}). Let   $\Omega \Subset \C^n$ be a bounded strongly $m$-pseudoconvex domain and $\mu$ a positive Borel measure on $\Omega$ satisfying the condition (\ref{eq:boundedsubsol}). Then for any $g \in \mathcal C^0 (\partial \Omega)$, there exists a unique $U = U_{g,\mu} \in \mathcal{SH}_m (\Omega)\cap L^{\infty} (\Omega)$ such that $(dd^c U)^m \wedge \beta^{n -m} = \mu$ on $\Omega$ and  $U|_{\partial \Omega} \equiv g.$
\end{theorem}

\subsection{The viscosity comparison principle}

In order to prove Theorem A, we will need  to prove an important result (Theorem \ref{thm:obstacle}). The proof of this result uses the viscosity comparison principle which was established for complex Hessian equations by H.C. Lu (\cite{Lu13}) in the spirit of the earlier work by P. Eyssidieux, V. Guedj and the second author on complex Monge-Amp\`ere equations (\cite{EGZ11}).

To state this comparison principle we need some definitions. 

Let $\Omega \Subset \C^n$ be   a bounded domain and $F : \Omega \times \R \longrightarrow \R$ a continuous function {\it non-decreasing} in the last variable.
\begin{definition}
Let $u: \Omega\rightarrow \R\cup\{-\infty\}$ be a function and $q$ be a $\mathcal C^2$ function in a neighborhood of $z_0\in \Omega.$ We say that $q$ touches $u$ from above (resp. below) at $z_0$ if $q(z_0)=u(z_0)$ and $q(z)\geq u(z)$ (resp. $q(z)\leq u(z)$) for every $z$ in a neighborhood of $z_0.$
\end{definition}
\begin{definition}\label{def: viscosity subsolution} 
An upper semicontinuous function $u: \Omega\rightarrow \R$ is  a viscosity subsolution to the equation
\begin{equation}\label{eq: heq 1}
(dd^c u)^m\wedge \beta^{n-m} = F(z,u)\beta^n,
\end{equation}
if for any $z_0\in \Omega$ and any $\mathcal C^2$ function $q$ which touches $u$ from above at $z_0$ then 
$$
\sigma_m(q) \geq F(\cdot,q (z_0))\beta^n, \ \text{at} \ z = z_0.
$$
 We will also say that $\sigma_m(u)\geq F(\cdot,u)\beta^n$ in the viscosity sense at $z_0$ and $q$ is an upper test function for $u$ at $z_0$.
\end{definition}
\begin{definition}\label{def: viscosity supersolution}
A lower semicontinuous function $v : \Omega \rightarrow \R$ is  a viscosity supersolution to (\ref{eq: heq 1}) if for any $z_0\in X$ and any $\mathcal C^2$ function $q$ which touches $v$ from below at $z_0$,
$$
[(dd^c q)^m\wedge \beta^{n-m}]_+\leq F(z,q)\beta^n, \ \text{at} \ z = z_0.
$$ 
Here $[\alpha^m\wedge\beta^{n-m}]_+$ is defined to be $\alpha^m\wedge\beta^{n-m} $  if $\alpha$ is $m$-positive and $0$ otherwise.  We will also say that $\sigma_m(v)_+ \leq F(\cdot,v)\beta^n$ in the viscosity sense at $z_0$ and $q$ is a lower test function for $v$ at $z_0$.
\end{definition}

\begin{remark} If $v \in \mathcal C^2 (\Omega)$ then $\sigma_m(v)\geq  F(z,v)\beta^n$ (resp. $[\sigma_m(v)]_+\leq F(z,v)\beta^n$) holds  on $\Omega$ in the viscosity sense iff it holds in the usual sense.
\end{remark}

\begin{definition}
A continuous function $u: \Omega \rightarrow \R$ is a viscosity solution to (\ref{eq: heq 1}) if it is both a subsolution and a supersolution.
\end{definition}

The first important result in this theory compares  the viscosity and potential subsolutions. 

\begin{proposition}[\cite{Lu13}] \label{prop: viscosity vs potential general case}
 Let $u$ be a bounded upper semi-continuous  function in $\Omega.$ Then the inequality 
\begin{equation}\label{eq: viscosity vs potential 2}
\sigma_m (u)\geq F(\cdot,u)\beta^n
\end{equation}
holds in the viscosity sense on $\Omega$ if and only if $u$ is $m$-subharmonic  and (\ref{eq: viscosity vs potential 2}) holds in the potential sense on $\Omega$.
\end{proposition}

Now we can state the viscosity comparison principle.

\begin{theorem} [\cite{Lu13}]\label{thm: viscosity comparison principle}
 Let $u : \Omega \longrightarrow \R$ be a bounded viscosity subsolution and $v : \Omega \longrightarrow \R$ be a viscosity supersolution of the equation
$$
\sigma_m(u)=F(\cdot,u)\beta^n,
$$ 
on $\Omega$. If $u \leq v$ on $\partial \Omega$  then $u\leq v$ on $\Omega.$
\end{theorem}

 For more details on this theory we refer to \cite{Lu13} and \cite{EGZ11} in the complex case and to \cite{CIL92} for the real case.
 
\subsection{Weak stability estimates}

An important tool in dealing with our problems is the notion of capacity. This was introduced by  Bedford and Taylor in their pionneering  work for the complex Monge-Amp\`ere operator (see \cite{BT82}). 
Let us recall the coresponding notion of capacity we will use here (see \cite{Lu12}, \cite{SA13}). Let $\Omega \Subset \C^n$  be a strongly $m$-pseudoconvex domain.  The $m$-Hessian capacity  is defined as follows. For any compact set $K \subset \Omega$,
$$
 \mathrm{Cap}_m(K,\Omega) := \sup \{\int_K  (dd^c u)^m \wedge \beta^{n - m} ; u \in \mathcal{SH}_m (\Omega) , - 1 \leq u \leq 0\}.
$$

We can extend this capacity as an outer capacity on $\Omega$. Given a  set $S \subset \Omega$, we define the inner capacity of $S$ by the formula
$$
\mathrm{Cap}_m(S,\Omega) := \sup \{\mathrm{Cap}_m(K,\Omega) ; K \, \, \hbox{compact} \, \, K \subset S\}.
$$ 

The outer capacity of $S$ is defined by the formula 
$$
\mathrm{Cap}^*_m(S,\Omega) := \inf \{\mathrm{Cap}_m(U,\Omega) ; U \, \, \hbox{ is open} \, \, U \supset S\}, 
$$ 

It is possible to show that $\mathrm{Cap}^*_m(\cdot,\Omega)$ is a Choquet capacity and then  any Borel set $ B \subset \Omega$ is capacitable and
for any compact set $K \subset \Omega$, 
 \begin{equation} \label{eq:cap}
 \text{Cap}_m(K,\Omega)=\int_{\Omega}(dd^c u_K^*)^m\wedge\beta^{n-m},
 \end{equation}
  where $u_K$ is the relative equilibrium potential of $(K,\Omega)$ defined by the formula :
  
 $$
 u_K:=\sup\{u\in \mathcal{SH}_m(\Omega) \, ; \,  u \, \leq \, -{\bf 1}_K  \, \mathrm{on } \, \, \Omega\},
 $$
 and $u_K^*$ is its upper semi-continuous regularization on $\Omega$ (see \cite{Lu12}).
 
 It is well known that $u_K^*$ is $m$-subharmonic on $\Omega$, $- 1 \leq u_K^* \leq 0$, $u_K^*= - 1$ quasi-everywhere (with respect to $\text{Cap}_m$) on $\Omega$ and $u_K^* \to 0$ as $z \to \partial \Omega$ (see \cite{Lu12}).

We will use the following definition. 
\begin{definition}  \label{def:cap-domination} Let $\mu$ be  a positive Borel measure on $\Omega$ and let $A,\tau > 0$ be positive numbers.  We say that $\mu$ is  dominated by the $m$-Hessian capacity with parameters $(A,\tau)$ if for any compact subset $K \subset \Omega$  with $\mathrm{Cap}_m (K,\Omega) \leq 1$,
\begin{equation} \label{eq:capdomination}
 \mu (K) \leq A \text{Cap}_m (K,\Omega)^{\tau}. 
\end{equation}
 \end{definition}
 Observe that by capacitability, this inequality is then satisfied for any Borel set $K \subset \Omega$.

 Let us mention that S.  Ko\l odziej was the first to relate the domination  of the measure $\mu$ by the Monge-Amp\`ere capacity to the regularity of the solution to complex Monge-Amp\`ere equations (see \cite{Kol96}). 
 
 Using his idea,  Eyssidieux, Guedj and the second author were able to establish in \cite{EGZ09} a weak stability $L^1$-$L^{\infty}$ estimate for bounded  solutions to the Dirichlet problem for the complex Monge-Amp\`ere equation. This result is the main tool in deriving estimates on the modulus of continuity of solutions to  the complex Monge-Amp\`ere and Hessian equations.

The following examples are due to  Dinew and Ko\l odziej (see \cite{DK14}).
\begin{example}
1. Dinew and Kolodziej  proved in \cite{DK14} that the volume measure $\lambda_{2 n}$ is dominated by capacity. Namely for any $1 < r < \frac{m}{n - m}$, there exists a constant $N (r) > 0$ such that for any compact subset $K \subset \Omega$, 
\begin{equation} \label{eq:DK}
\lambda_{2 n} (K) \leq N (r) \mathrm{Cap}_m (K,\Omega)^{1+ r}.
\end{equation}
 Observe that this estimate is sharp in terms of the exponent when $m < n$. This can be seen by taking $\Omega = \B$ the unit ball and  $K := \bar{\B_s} \subset \B$ the closed ball of radius $s \in ]0,1[$, since $\mathrm{Cap}_m(\bar{\B}_s ,\B) \approx  s^{2 (n-m)}$ as $s \to 0$ (see \cite{Lu12}).
 When $m=n$ we know that the domination is much more precise (see  \cite{ACKPZ09}).
 
2. Let $ 0 \leq f \in L^p (\Omega)$ with $p > n \slash m$. Then $ \frac{n (p-1)}{p (n - m)} > 1$. By H\"older inequality and inequality (\ref{eq:DK}) we obtain: for any  $1 < \tau < \frac{n (p-1)}{p (n - m)}$ there exists a constant $M (\tau) > 0$ such that for any compact set $K \subset \Omega$, 
\begin{equation} \label{eq:DK2}
\int_K f d \lambda_{2 n} \leq  M (\tau) \Vert f\Vert_p  \mathrm{Cap}_m (K,\Omega)^{\tau}.
\end{equation}

\end{example}

Theorem A will provide us with many new examples.
 
 The condition (\ref{eq:capdomination}) plays an important role in the following stability result which will be a crucial point in the proof of our theorems (see \cite{EGZ09, GKZ08, Ch16}).

\begin{proposition} \label{prop:stability}  Let $\mu$ be a positive Borel measure on $\Omega$ dominated by the $m$-Hessian capacity  with parameters $(A, \tau)$ such that  $\tau > 1$. 

 Then for  any  $u, v \in \mathcal {SH}_m (\Omega) \cap L^{\infty} (\Omega)$ such that $(dd^c u)^m \wedge \beta^{n - m} \leq \mu$ on $\Omega$ and $\liminf_{\partial \Omega} (u - v) \geq 0$, we have  

\begin{equation} \label{eq:stability}
\sup_{\Omega}  (v - u)_+ \leq 2 \Vert (v-u)_+\Vert_{1,\mu}^{1 \slash (m+1)} + C \|(v-u)_+\|^\gamma_{1,\mu},
\end{equation}
where  $\Vert (v-u)_+ \Vert_{1,\mu} := \int_{\Omega} (v-u)_+ d \mu$ and 
\begin{equation} \label{eq:stbilityconstant}
C := 1 +  \frac{2^{\tau} A^{\frac{1}{m}}}{1-2^{1 -\tau}}, \,  \gamma = \gamma (\tau,m):= \frac{\tau - 1}{\tau (m + 1) - m }\cdot 
\end{equation}
\end{proposition}
 Observe that the most relevant case in applications is when $\Vert (v-u)_+ \Vert_{1,\mu}$ is small. So the right exponent is $\gamma < 1 \slash (m +1)$.
 \begin{proof}
 The proof uses an idea which goes back to Ko\l odziej (\cite{Kol96}) with some simplifications due to Guedj, Eyssidieux and the second author (see \cite{EGZ09, GKZ08}). It  relies on the following estimates : for any $t > 0, s > 0$

 \begin{equation}\label{eq:Cap-MA}
 t^m {Cap}_m(\{u< v -s-t\},\Omega)\leq \int_{\{u<v -s\}}(dd^cu)^m\wedge\beta^{n-m}.
\end{equation}

Indeed let $t > 0, s > 0$ fixed and $ w \in  \mathcal{SH}_m(\Omega)$ be given such that $-1 \leq w \leq0$. Then 
 $$\{u-v<-s-t\}\subset\{u-v< t w -s\} \subset\{u-v<-s\} \Subset\Omega.$$
 
 It follows that 
\begin{eqnarray*}
 t^m \int_{\{u-v<-s-t\}}(dd^c w)^m\wedge\beta^{n-m}&\leq&\int_{\{u< v - s-t\}}(dd^c(v+tw))^m\wedge\beta^{n-m}\\
 &\leq& \int_{\{u< v+ tw - s\}}(dd^c(v+tw))^m\wedge\beta^{n-m}.
 \end{eqnarray*}
On the other hand the comparison principle yields
\begin{eqnarray*}
 \int_{\{u< v+ tw - s\}}(dd^c(v+tw))^m\wedge\beta^{n-m} &\leq&\int_{\{u < v +t w - s\}}(dd^c u)^m\wedge\beta^{n-m}\\
 &\leq&\int_{\{u< v -s\}}(dd^cu)^m\wedge\beta^{n-m}.
\end{eqnarray*}
The last two inequalities imply  (\ref{eq:Cap-MA}).

Applying inequality (\ref{eq:Cap-MA}) with the parameter $(s \slash 2,  s\slash 2)$ instead of $(t,s)$ and taking into acount that $u$ is a supersolution, we obtain

\begin{eqnarray} 
\nonumber \mathrm{Cap}_m(\{u< v - s \},\Omega) & \leq & 2^m s^{-m} \int_{\{u<v - s\slash 2\}}(dd^cu)^m\wedge\beta^{n-m} \\
& \leq & 2^{m + 1} s^{-m - 1} \int_\Omega (v-u)_+ d \mu. 
\end{eqnarray}
Set $s_0 :=  2 \Vert (v-u)_+\Vert_{1,\mu}^{1 \slash (m+1)}$.  Then for any $s \geq s_0$, 

 \begin{equation} \label{eq:smallcap}
\mathrm{Cap}_m(\{u< v - s \},\Omega)\leq 1.
 \end{equation}

Fix $\varepsilon > 0$ and $s \geq 0$. Then applying inequality (\ref{eq:Cap-MA}) with $s_0 + s + \varepsilon$ instead of $s$ and taking into account the fact that   $(dd^c u)^m \wedge \beta^{n - m} \leq \mu$ weakly on $\Omega$,  we get
\begin{equation} \label{eq:cap-mu}
t^m {Cap}_m(\{u< v  - s_0 - \varepsilon - s -t\},\Omega)\leq \int_{\{u<v -s_0  - \varepsilon - s\}} d \mu.
\end{equation}
Set $ f (s) = f_\varepsilon (s):={Cap}_m(\{u-v<-s - s_0 -\varepsilon\},\Omega)^{\frac{1}{m}}$.  By  (\ref{eq:smallcap}), we  have $f (s) \leq 1$. Hence since $\mu$ is dominated by capacity,  it follows that for any $t > 0$ and $s > 0$,

 $$
 t f (s + t) \leq A^{\frac{1}{m}} f (t)^{1 + a}, \, \, \hbox{where} \, \, a := \tau - 1 > 0.
 $$ 
 
It follows from \cite[Lemma 2.4]{EGZ09}) that $f (s) = 0$ for any $s\geq S_\infty$ where 
 $$
 S_\infty:=\frac{2 A^{\frac{1}{m}} }{1-2^{-a}}  [f (0)]^{a},
 $$
 Thus $v-u\leq s_0 + \varepsilon + S_\infty$ quasi everywhere on $\Omega$ and then the inequality holds everywhere on $\Omega$ i.e.
 $$
 \max (v-u)_+ \leq s_ 0 + \varepsilon + \frac{2 A^{\frac{1}{m}}}{1-2^{-a}} {Cap}_m(\{v - u >\varepsilon\},\Omega)^{a}
 $$
 Applying  (\ref{eq:Cap-MA}) with $t=\varepsilon$ and $s = 0$ we obtain
 $$
 {Cap}_m(\{v - u > \varepsilon\},\Omega)\leq 2 \varepsilon^{-m-1}\|(v-u)_+\|_{1,\mu}.
 $$
 
 As a consequence of the previous estimate, we obtain
 $$
 \sup_\Omega(v-u)\leq 2 \Vert (v-u)_+\Vert_{1,\mu}^{1 \slash (m+1)} +  \varepsilon+ C' \varepsilon^{-a (m+1)}\|(v-u)_+\|^{a}_{1,\mu},
 $$
 where $C' :=  \frac{2^{a + 1} A^{\frac{1}{m}}}{1-2^{-a}}$.
 Set  $\varepsilon :=\|(v-u)_+\|^\gamma_{1,\mu}$, with $\gamma:=\frac{a}{1 +a(m+1)} = \frac{\tau - 1}{(\tau - 1) (m + 1) + 1}.$ Then
 $$
\sup_{\Omega}(v-u)_+\leq 2 \Vert (v-u)_+\Vert_{1,\mu}^{1 \slash (m+1)} + C \|(v-u)_+\|^\gamma_{1,\mu},
 $$
 where $C := C' + 1 = 1 + \frac{2^{a + 1} A^{\frac{1}{m}}}{1-2^{-a}} = 1 +  \frac{2^{\tau} A^{\frac{1}{m}}}{1-2^{1 -\tau}}$.
\end{proof}

\section{Subharmonic envelopes and obstacle problems } 
 Here we prove some results that will be used in the proof of the Theorem A. Since they are  of independent interest, we will state them in the most general form and give complete proofs. 

\subsection{Subharmonic envelopes } 

Let $\Omega \Subset \C^n$ and  $h : \Omega \longrightarrow \R$ is a non positive bounded Borel function and  define the corresponding projection:
\begin{equation} \label{eq:subextension}
\tilde h = P_{m,\Omega}  (h) := \left(\sup \{v \in \mathcal {SH}_m (\Omega) ;  \,  v \leq h \, \text{in} \, \, \Omega\}\right)^*. 
\end{equation}
Observe that we do not need to take the upper semi-continuous regularization if $h$ is upper semi-continuous on $\Omega$. On the other hand, we can easily see that
$$
 P_{m,\Omega}  (h) := \sup \{v \in \mathcal {SH}_m (\Omega) ;  \,  v \leq h \, \text{quasi everywhere on} \, \, \Omega\}, 
$$
where $v \leq h$ quasi everywhere on $\Omega$ means that the exceptional set where $v \geq h$ has zero  $Cap_m$-capacity.

This is a classical construction in Potential Theory and has been considered in Complex Analysis  first by H. Bremermann in \cite{Brem59},  J.B. Walsh in \cite{Wal69} and also by J. Siciak in \cite{Sic81}. Later it has been studied by Bedford and Taylor when solving the Dirichlet problem for the  the complex Monge-Amp\`ere equation (\cite{BT76}, \cite{BT82}. In the setting of compact K\"ahler manifolds it has bee considered R. Berman and J.-P. Demailly  in \cite{BD12} and later in \cite{Ber19}. It has been also considered recently in \cite{GLZ19} in connexion with the supersolution problem for complex Monge-Amp\`ere equations,  where a precise estimate of its complex Monge-Amp\`ere measure was given. 

We will extend these last results to Hessian equations.

 \begin{lemma}  \label{lem:projection} Let  $\Omega \Subset \C^n$ be a bounded strongly $m$-pseudoconvex domain and $h$  a bounded lower semi-continuous function on $\Omega$. Then the function  $\tilde h :=  P_{m,\Omega}  (h)$  satisfies the following properties:

$(i)$ $\tilde h  \in \mathcal {SH}_m (\Omega) \cap L^{\infty} (\Omega)$,  and $\tilde h \leq h$ a.e. on $\Omega$;

$(ii)$ if $h$ is continuous on $\bar \Omega$, then $\tilde h$ is continuous on $\bar{\Omega}$ %with a modulus of continuity $\kappa_{\tilde h} \leq \kappa_h$
 and satisfies the following properties
\begin{equation} \label{eq:boundaryvaluesh}
\lim_{\Omega \ni z \to \zeta} \tilde h (z) = h (\zeta), \, \, \zeta \in \partial \Omega,
\end{equation} 

$(iii)$ $ \int_\Omega (\tilde h - h) (dd^c \tilde h)^m \wedge \beta^{n - m}   = 0$.

\end{lemma}

\begin{proof}
Observe that    $\min_{\bar \Omega} h \leq \tilde h \leq \max_{\bar \Omega} h$ on $\Omega$.  
1. Property $(i)$ follows from the general theory (see \cite{Lu12}). 

2. Property $(ii)$ can be proved using the perturbation method due to  J.B. Walsh (see \cite{Wal69}). Let us recall the argument for completeness. 

We first prove that $\tilde  h$ satisfies (\ref{eq:boundaryvaluesh}) meaning that it has boundary values equal to $h$ and then it extends as a function on $\bar \Omega$ which is continuous on $\partial \Omega$. Indeed fix $\varepsilon > 0$ and  let $h'$ be a $C^2$ approximating function on $\bar \Omega$ such that $h - \varepsilon \leq h' \leq h $  on $\bar \Omega$.  Let $\rho$ be the strongly $m$-subharmonic defining function for $\Omega$. Then there exists a constant $A > 0$ such that $u := A \rho + h'$ is $m$-subharmonic on $\Omega$ and $u \leq h' \leq h$ on $\bar \Omega$. Then by definition of the envelope,  we have $u \leq \tilde h \leq h$ on $\bar \Omega$. Therefore for any $\zeta \in \partial \Omega$, 
\begin{eqnarray*}
h (\zeta) - \varepsilon \leq  h' (\zeta) & = & \lim_{\Omega \ni z \to \zeta } u (z)   \\
& \leq & \liminf_{\Omega \ni z \to \zeta } \tilde h (z) \leq  \limsup_{\Omega \ni z \to \zeta } \tilde h (z)  \leq  h (\zeta).
\end{eqnarray*}
Since $\varepsilon > 0$ is arbitrary, we obtain the identity (\ref{eq:boundaryvaluesh}).
We can then extend $\tilde h$ to $\bar \Omega$ by setting $\tilde h (\zeta) = h (\zeta)$ for $\zeta \in \partial \Omega$.  To prove the continuity of $\tilde h$ on $\bar \Omega$, we use the perturbation argument of J.B. Walsh. Fix  $\delta > 0$ small enough, $a \in \C^n$ such that $\vert a \vert \leq \delta$ and set $\Omega_a := (- a) + \Omega$.

We define the  modulus of continuity of $\tilde{h}$ near the boundary as follows:
$$
\tilde{\kappa}_{\tilde h} (\delta) := \sup \{ \vert \tilde h(z) - \tilde{h} (\zeta)\vert  \, ; \, z \in \Omega, \zeta \in \partial \Omega, \vert z - \zeta\vert \leq \delta. 
$$
Then since $\tilde h = h$ is uniformly continuous on $\partial \Omega$, we see that $\lim_{\delta \to 0^+} \tilde{\kappa}_{\tilde h} (\delta) = 0$. By definition of  $\tilde{\kappa}_{\tilde h}$,  for any $z \in  \Omega \cap \partial \Omega_a$, we have  
$$\tilde{ h} (z + a)  \leq \tilde h (z) +  \tilde{\kappa}_{\tilde{h}} (\delta) \leq \tilde h (z) +  \tilde{\kappa}_{\tilde{h}} (\delta)  + \kappa_h (\delta),
$$ 
where $ \kappa_h (\delta)$ is the modulus of continuity of $h$ on $\bar{\Omega}$.

Therefore by the gluing principle, the function defined by 

$$
 v (z):=\left\{\begin{array}{lcl}
\max \{ \tilde h (z) , \tilde h (z + a) - \tilde{\kappa}_{\tilde{h}} (\delta) - \kappa_h (\delta)\} \, &\hbox{ if}\ z \in\Omega \cap \Omega_a \\
 \tilde h (z)  &\hbox{if}\  z \in\Omega\setminus\Omega_a\\
\end{array}\right.
$$
is $m$-subharmonic on $\Omega$ and satisfies $v \leq h$ on $\bar \Omega$. Therefore $v \leq \tilde h$ on $\bar \Omega$ and then 
$$
\tilde h (z + a) - \tilde{\kappa}_{\tilde{h}} (\delta) -  \kappa_h (\delta) \leq \tilde h (z),
$$
 for any $z \in \Omega \cap \Omega_a$ with $\vert a\vert \leq \delta$. This proves that $\tilde{h}$ is uniformly continuous on $\bar{\Omega}$.

3. Property $ (iii)$  follows by a standard balayage argument in Potential Theory which goes back to Bedford and Taylor for the complex Monge-Amp\`ere equation (\cite{BT76}, \cite{BT82}, see also \cite{GLZ19}).
  \end{proof}
  
 \begin{remark}
 The proof above does not give any information on the modulus of continuity of $\tilde h$ in terms of the modulus of continuity of $h$. In other words we do not know if $\tilde{\kappa}_{\tilde h}$ is comparable to $\kappa_h$.
 
 However if $h$ is $C^2$-smooth on $\bar \Omega$, the function $u := A \rho + h$, considered in the proof above with $h'=h$,  is $m$-subharmonic on $\Omega$, Lipschitz on $\bar{\Omega}$ and safisfies $u \leq \tilde{h} \leq h$ on $\bar{\Omega}$.
 Then this implies that $\tilde{\kappa}_{\tilde{h}} (\delta)  \leq  \kappa_h (\delta) + \kappa_u (\delta) \leq C \kappa_h (\delta)$, where $C > 0$ is a uniform constant. 
 Therefore the modulus of continuity of $\tilde h$ satisfies the inequality $\kappa_{\tilde h}(\delta) \leq C' \kappa_h(\delta)$,
 where $C' > 0$ is an absolute constant.

 This information is not needed here, but it is worth mentioning that this an interesting open problem related to the regularity of solutions to obstacle problems.
 We will come back to this in a subsequent work.
 
 \end{remark}
 \subsection{An obstacle problem}

 \begin{theorem} \label{thm:obstacle}
Let $h \in \mathcal C^2 ({\bar{\Omega}})$. Then $\tilde h := P_{m,\Omega} h  \in \mathcal{SH}_m (\Omega) \cap C^0 (\bar \Omega)$ and its   $m$-Hessian measure  satisfies the following inequality :
\begin{equation} \label{eq:BerIneq1}
 (dd^c \tilde h)^m \wedge \beta^{n - m}  \leq {\bf 1}_{\{   \tilde h = h\}}  \sigma_m^+ (h),
\end{equation}
in the sense of currents on $\Omega$.
 \end{theorem}
 Here for a  function $h \in \mathcal C^2 ({\bar{\Omega}})$, we set 
$$
\sigma_m^+ (h) := {\bf 1}_G  \, \sigma_m (h),
$$
pointwise on $\Omega$,  where $G$ is the set of points  $z \in \Omega$ such that  $dd^c h (z) \in \Theta_m$ i.e. the $(1,1)$-form $dd^c h (z)$ is $m$-positive (see Definition  \ref{def:mpositive}).

 \begin{proof} To prove (\ref{eq:BerIneq1}), we proceed as in  \cite{GLZ19},  using an idea which goes back to R. Berman \cite{Ber19}. 
 
 Thanks to the property 
 $(ii)$ of Lemma \ref{lem:projection},  it is enough to prove that
 \begin{equation} \label{eq:BerIneq2}
 (dd^c \tilde h)^m \wedge \beta^{n - m}  \leq \sigma_m^+ (h),
\end{equation}
in the sense of currents on $\Omega$.

We procced in two steps:

  1) Assume first that $\Omega$ is smooth strongly $m$-pseudoconvex and $h \in \mathcal{C}^2 (\bar \Omega)$ and  consider the following Dirichlet problem for the complex $m$-Hessian equation depending on the parameter $j \in \N$,
 \begin{equation} \label{eq:BerEqu}
 (dd^c u)^m \wedge \beta^{n - m} = e^{j (u-h)} \sigma_m^+ (h), \, \, u = h \, \, \mathrm{in} \, \, \, {\partial \Omega}.
 \end{equation}
 
 By \cite{Lu13}, for each $j \in \N$, there exists a unique continuous solution $u_j \in \mathcal {SH}_m (\Omega) \cap\mathcal C^0  (\Omega)$ to this problem (see also \cite{Ch16}).
 
 Our goal is to prove that the sequence  $(u_j)_{j\in \N}$ increases to $ \tilde  h$ uniformly  on $\bar{\Omega}$.  We argue as in \cite{GLZ19} with obvious modifications. Recall $h$ is $C^2$ in $\bar{\Omega}$. Then by definition $h$ is a viscosity supersolution to the Dirichlet problem (\ref{eq:BerEqu}). Moreover by Proposition \ref{prop: viscosity vs potential general case}, $u_j$ is a viscosity subsolution to the Dirichlet problem (\ref{eq:BerEqu}).  By the viscosity comparison principle Theorem \ref{thm: viscosity comparison principle}, we conclude that $u_j \leq h$ in $ \Omega$ since $u_j = h$ on $\partial \Omega$.
  
 Therefore the pluripotential comparison principle  Proposition \ref{prop:Comparison Principle} implies  that $(u_j)$  is an increasing sequence.
 On the other hand, by Theorem \ref{thm:boundedsubsolution} there exists a bounded $m$-subharmonic function  $\psi$ on $\Omega$ which is a solution to the complex Hessian equation 
 $$
 \sigma_m (\psi) =  e^{\psi - h} \sigma_m^+ (h)
 $$ 
 with $\psi = h$ on $\partial \Omega$. 
 Moreover  for any $j \in \N$, one can easily  check that the function defined by the formula
 $$
 \psi_j := (1- 1 \slash j) \tilde  h + (1 \slash j) ( \psi - m \log j) 
 $$ 
  is a (pluripotential) subsolution to the equation (\ref{eq:BerEqu}), since $ \tilde  h \leq h$ on $\Omega$. Hence by Proposition \ref{prop:Comparison Principle} we have $\psi_j \leq u_j$ on $\Omega$.
 
 Summarizing we have proved that for any $j \in \N$, $\psi_j \leq u_j \leq \tilde  h  $ on $\Omega$. 
 Therefore $ 0 \leq \tilde  h - u_j \leq \tilde  h - \psi_j =  (1 \slash j) ( \tilde  h - \psi + m \log j) $ on $\Omega$ for any $j \in \N^*$.
 This proves that $u_j$ converges to $\tilde  h$ uniformly  on $\Omega$. Then since $u_j \leq h$ on $\Omega$,  taking the limit as $j \to + \infty$ in (\ref{eq:BerEqu}) we obtain inequality (\ref{eq:BerIneq2}) by the continuity of the Hessian operators for uniform convergence (see \cite{Lu12}).

2) For the general case of a bounded $m$-hyperocnvex domain, we approximate $\Omega$ by an increasing sequence $(\Omega_j)_{j \in \N}$ of smooth strongly $m$-pseudoconvex domains such that for any $j \in \N$, $\Omega_{j + 1} \subset \Omega_j$ and $\Omega = \cup_{j \in \N} \Omega_j$.
Then it is easy to see that the sequence $(P_{m,\Omega} h_j)$ decreases to $P_{m,\Omega} h$ on $\Omega$ (see \cite{GLZ19}). Thus the result follows from the previous case by the continuity of the Hessian operator for monotone sequences.
 \end{proof}

 It's worth mentioning that these envelopes have been considered by several authors in the context of compact K\"ahler manifolds. When $h$ is $C^2$ it was proved recently that $P (h)$ is $C^{1,1}$ (see \cite{ChZh17}, \cite{T18}, \cite{Ber19}) and  equality holds in (\ref{eq:BerIneq1}), which means that $P(h)$ is a solution to an obstacle problem (see \cite{BD12}). 

We can address a similar question. 

\smallskip
 
 {\it Question :}  Is it true that $\tilde h$ is $C^{1,1}$ locally on $\Omega$ when $h$ is $C^2$ on $\bar \Omega$  ? Is there equality in (\ref{eq:BerIneq1}) ?

 \begin{corollary} \label{cor:smoothing} Let $\Omega \Subset \C^n$ be a strongly $m$-pseudconvex domain.  Let $u \in \mathcal {SH}_m (\Omega)$ a negative $m$-subharmonic function. Then there exists a decreasing sequence $(u_j)$ of continuous $m$-subharmonic functions on $\Omega$ with boundary values $0$ which converges pointwise to $u$ on $\Omega$. 
\end{corollary}

\begin{proof} We can assume that $u$ is bounded on $\Omega$ and  extend it as a semi-continous function on $\bar \Omega$. Let $(h_j)_{j \in \N}$ be a decreasing sequence of smooth functions in a neighbourhood of $\bar \Omega$ which converges to $u$ in $\bar \Omega$.
For each $j \in \N$, consider the $m$-subharmonic envelope  
$v_j := P_{\Omega} h_j$ on $\Omega$ and set $u_j := \max \{v_j, j \rho\}$ on $\Omega$, where $\rho$ is a continuous  $m$-subharmonic defining function for $\Omega$. Then by Lemma \ref{lem:projection}, by the Lemma $(u_j)$ is a decreasing sequence of continuous $m$-subharmonic functions on $\Omega$ which converges to $u$ on $\Omega$.
\end{proof}

Applying the smoothing method of Richberg it is possible to construct a decreasing sequence of smooth $m$-subharmonic functions on $\Omega$ which converges to $u$ in $ \Omega$ (see \cite{P14}).

\section{Hessian measures of H\"older continuous potentials }

 In this section we will prove two important results which will be used in the proof of the main theorems stated in the introduction. 
 
\subsection{Hessian mass estimates near the boundary}

Here we prove a comparison inequality which seems to be new even in the case of a complex  Monge-Amp\`ere measure. 

\begin{lemma} \label{lem:ComparisonIneq}
Let $\Omega \Subset \C^n$ be a bounded strongly $m$-pseudoconvex domain and  $\varphi \in \mathcal{SH}_m(\Omega)\cap \mathcal{C}^{\alpha} (\bar{\Omega})$ ($0 < \alpha \leq 1$) with $\varphi \equiv 0$ on $\partial \Omega$. Then for any Borel set   $K \subset \Omega$,  we have
$$
\int_K (dd^c\varphi)^m\wedge\beta^{n-m}  \leq L^m \left[\delta_K (\partial \Omega)\right]^{ m \alpha} \, \mathrm{Cap}_m (K,\Omega),
$$
where 
$$
\delta_K (\partial \Omega) := \sup_{z \in K} \mathrm{dist} (z ; \partial \Omega)
$$
and $L > 0$ is the H\"older norm of $\varphi$.
\end{lemma}
The constant $\delta_K (\partial \Omega)$ is the Hausdorff distance of $K$ to the boundary in the sense that $\delta_K (\partial \Omega) \leq \varepsilon $ means that $K$ is contained in the $\varepsilon$-neighbourhood of $\partial \Omega$.

The relevant point here is that the estimate takes care of the behaviour at the boundary. It shows in particular that if the volume of the compact set is fixed, the capacity tends to  $+ \infty$ when the compact set approaches the boundary at a rate controlled by the Hausdorff distance of the compact to the boundary. 
\begin{proof} 
By inner regularity, we can assume that $K \subset \Omega$ is compact. 
 Since  $\varphi$ is H\"older continuous on $\bar \Omega$,  we have 
$\varphi (\zeta) - \varphi (z) \leq L \vert \zeta - z\vert^{\alpha}$ for any $\zeta \in \partial \Omega$ and any  $z \in \Omega$. 

Fix a compact set $K \subset \Omega$. Since $\varphi = 0$ in $\partial \Omega$, it follows that for any $z \in K$,
$$
- \varphi (z) \leq \kappa \left[\mathrm{dist} (z,\partial \Omega)\right]^{\alpha} \leq L \left[\delta_K (\partial \Omega)\right]^{\alpha} =:a.
$$
Therefore the function $v := a^{-1} \varphi \in \mathcal{SH}_m (\Omega)$ and $ v \leq 0$ on $\Omega$ and $v \geq - 1$ in $K$.   Fix $\varepsilon >0$ and let $u_K$ be the relative extremal $m$-subharmonic function of $(K,\Omega)$.  Then $ K \subset \{ (1 + \varepsilon) u_K^* < v\} \cup \{u_K < u_K^*\}$. Since the set $\{u_K < u_K^*\}$ has zero $m$-capacity (see \cite{Lu12}), it follows from  the comparison principle that for any $\varepsilon > 0$,
\begin{eqnarray*}
\int_K (dd^c v)^m \wedge \beta^{n - m}  &\leq &\int_{\{ (1 + \varepsilon) u_K^*  < v\}} (dd^c v)^m \wedge \beta^{n - m} \\
&\leq & (1+ \varepsilon)^m \int_{\{ (1 + \varepsilon) u_K^* < v\}} (dd^c u_K^*)^m \wedge \beta^{n - m} \\
& \leq & (1+ \varepsilon)^m \text{Cap}_m (K,\Omega).
\end{eqnarray*}
The lsat inequality follows from (\ref{eq:cap}).
The estimate of the Lemma follows by letting $\varepsilon \to 0$.
\end{proof}

\subsection{H\"older continuity  of Hessian measures} 

In order to prove the H\"older continuous subsolution theorem we need an additional argument following an idea which goes back to \cite{DDGKPZ15} and used in a systematic way in \cite{N18a} (see also \cite{KN19}).

Given a continuous function $g \in \mathcal{C}^0 (\partial \Omega)$ and a real number $R > 0$, we denote by $\mathcal{E}^g_m(\Omega,R)$ the convex set of bounded $m$-subharmonic functions $v $ on $\Omega$ such that $v = g$ on $\partial \Omega$ normalized by the mass condition  $\int_{\Omega} (dd^c v)^m \wedge \beta^{n -m} \leq R$.

In order to prove Theorem B, we will need the following lemma.

  \begin{lemma}\label{lem:ModC}
Let $\varphi\in \mathcal{E}^0_m (\Omega)\cap \mathcal{C}^{\alpha} (\overline\Omega)$, with $ 0 < \alpha \leq 1$.  Then there exists  a constant $C_k  = C (k,m,\varphi, \Omega) >0$ such that for any $0 < \delta < \delta_0$, and any $u,v\in \mathcal{SH}_m (\Omega,R)$ such that $u = v$ on $\partial \Omega$, we have for $1 \leq k \leq m$
\begin{equation} \label{eq:MocEst2}
\int_{\Omega}|u-v| (dd^c\varphi)^k\wedge\beta^{n-k}\leq C_k \, R \,  \left[\Vert u-v \Vert_1\right]^{\tilde \alpha_k },
\end{equation} 
 where  $\tilde \alpha_k   := (\alpha\slash 2)^k \slash m$, provided that $\Vert u-v \Vert_1 \leq 1$.

Moreover  if $g \in C^{1,1} (\partial \Omega)$, for any $1 \leq k\leq m$, there exists a constant $C'_k  = C' (k,m, \varphi,g,\Omega) > 0$  such that for any $0 < \delta < \delta_0$, and every $u,v\in \mathcal{SH}_m (\Omega,R)$ with $u = v = g$ on $\partial \Omega$, we have 

\begin{equation} \label{eq:MocEst1}
\int_{\Omega}|u-v|(dd^c\varphi)^k\wedge\beta^{n-k}\leq C'_k R  \left[\Vert u-v \Vert_1\right]^{\alpha_k},
\end{equation}
where  $\alpha_k := (\frac{\alpha}{ 2})^k$,  provided that $\Vert u-v \Vert_1 \leq 1$.
\end{lemma}

\begin{proof}
Recall the following notation for the complex Hessian measure of $\varphi$:
$$
\sigma_k{(\varphi)} := (dd^c\varphi)^k\wedge\beta^{n-k} \  \  \  1 \leq k\leq m.
$$

Observe that for any $\varepsilon > 0,$ $ u_\varepsilon := \max \{u - \varepsilon, v \} \in \mathcal{E}^g_m(\Omega)$ and $u_\varepsilon = v$ near the boundary $\partial \Omega$. By the  the comparison principle, this implies that 
$u_\varepsilon \in \mathcal{E}^g_m(\Omega,R)$.  Therefore, replacing $u$ by $u_\varepsilon$,  we can assume that $u \geq v$ on $\Omega$ and  $u = v$ near the boundary $\partial \Omega$.  Then the required estimates  will follow from this case since $\vert u - v \vert =   ( \max\{u,v\} - u) + (\max\{u,v\}  - v)$.

On the other hand by approximation on the support $S$ of $u-v$ which is compact, we can assume that $u$ and $ v $ are smooth on a neighbourhood of $S$.

We first extend $\varphi$ as a H\"older continuous function on $\C^n$. 
    Indeed recall that for any $z, \zeta \in \bar{\Omega}$, we have $\varphi (z) \leq \varphi (\zeta) + \kappa \vert z - \zeta\vert^\alpha $. Then it is easy to see that the following function
 \begin{equation} \label{eq:Holderextension}
   \bar{\varphi} (z) := \sup \{\varphi (\zeta) - \kappa \vert z - \zeta\vert^\alpha  ; \zeta \in \bar{\Omega}\}, \, \, z \in \C^n.
 \end{equation}
 is  H\"older continuous of order $\alpha$ on $\C^n$  and $ \bar{\varphi} = \varphi$ on $\Omega$. For simplicity, we will denote this extension by $\varphi$.   
 
   We approximate $\varphi$ by convolution and denote by $\varphi_{\delta}$ ($0 < \delta < \delta_0$) the  smooth approximants of  $\varphi $ on $\C^n$, defined  for $z \in \C^n$  by the formula 
 \begin{equation}   \label{eq:regularization2}
 \varphi_\delta (z) := \int_{\C^n} \varphi (\zeta) \chi_\delta(z - \zeta) d \lambda_{2n} (\zeta),
 \end{equation}
  where  $(\chi_\delta)_{\delta}$ is a smooth radial kernel approximating the Dirac unit mass at the origin. 
 
 Then by Lemma \ref{lem:Poisson-Jensen},  for $0 < \delta < \delta_0$, $\varphi_\delta \in  \mathcal{SH}_m (\Omega_{\delta})\cap\mathcal{C}^{\infty}(\mathbb C^n)$.
  To prove the required estimates, we will argue by induction on $0 \leq k \leq m$.  Fix $0 \leq k \leq m-1$ and  $\delta > 0$ and write
  $$
  \int_{\Omega}(u-v) (dd^c\varphi)^{k+1} \wedge\beta^{n-k-1} = A (\delta) + B (\delta), 
  $$
  where   
 $$
  A (\delta):=\int_{\Omega}(u-v)dd^c \varphi_{\delta}\wedge(dd^c\varphi)^k\wedge\beta^{n-k-1},
 $$
     and 
 \begin{eqnarray*}
     B (\delta) &:=& \int_{\Omega}(u-v) dd^c(\varphi - \varphi_{\delta} - \kappa \delta^\alpha)\wedge(dd^c\varphi)^k\wedge\beta^{n-k-1},
 \end{eqnarray*}
 where we recall that $ \varphi - \kappa \delta^\alpha \leq \varphi_\delta \leq \varphi + \kappa \delta^\alpha$ on $\bar{\Omega}$ by hypothesis.

The  first term $A (\delta)$ is estimated as follows. Observe that we have 
$dd^c \varphi_\delta \leq M_1 \kappa \delta^{\alpha -2}  \beta$ pointwise  on $\Omega$, where $M_1 > 0$ is a uniform  bound on the second derivatives of $\chi$. Then since $u \geq v$ we deduce that 
 \begin{equation} \label{eq:estimationA}
 \vert A (\delta) \vert \leq M_1 \kappa \delta^{\alpha -2} \int_{\Omega} (u-v) (dd^c \varphi)^k \wedge \beta^{n-k}.
 \end{equation}

We now estimate the second term $B (\delta)$.
Since $ u - v = 0$ near the boundary $\partial \Omega$ i.e. on $\Omega \setminus \Omega'$, where $\Omega' \Subset \Omega$ is an open set, we can integrate by parts to get the following formula
$$
B (\delta) = \int_{\Omega'} (\varphi_{\delta} - \varphi +   \kappa \delta^\alpha) dd^c (v-u) \wedge(dd^c\varphi)^k\wedge\beta^{n-k-1},
$$
 and then  since $0 \leq \varphi_{\delta}-\varphi + \kappa \delta^\alpha \leq 2 \kappa \delta^{\alpha}$ on $\Omega$, it follows that
$$
\vert B (\delta)\vert \leq  2 \kappa \delta^{\alpha} \int_{\Omega'} dd^c v \wedge(dd^c\varphi)^k\wedge\beta^{n-k-1}.
$$
Therefore, we get
  
 \begin{equation} \label{eq:estimationII}
  |B (\delta)| \leq 2 \kappa \delta^{\alpha} \,   I'_k(v,\varphi),
 \end{equation}
  where $I'_k (v,\varphi) := \int_{\Omega'} dd^c v \wedge(dd^c\varphi)^k\wedge\beta^{n-k-1}$.

 The problem is to estimate the terms $I'_k (v,\varphi) $ with a uniform constant which does not depend on  $\Omega' \Subset \Omega$. 
We could use the obvious inequality  
$$
 \int_{\Omega'} dd^c v \wedge(dd^c\varphi)^k\wedge\beta^{n-k-1} \leq \int_{\Omega} dd^c v \wedge(dd^c\varphi)^k\wedge\beta^{n-k-1},
 $$
to conclude   if we can show that $ I_k  (v,\varphi)  := \int_{\Omega} dd^c v \wedge(dd^c\varphi)^k\wedge\beta^{n-k-1}$ is finite, which is the case thanks to Lemma \ref{lem:Cegrell2}.

 \smallskip   
 \smallskip
 
Therefore by (\ref{eq:claim}) and (\ref{eq:estimationII}),  we get  
\begin{equation} \label{eq:estimation3}
  |B (\delta)| \leq  \kappa  \, d(m,n) \, R \, \delta^{\alpha},
 \end{equation}

 Combining  the inequalities (\ref{eq:estimationA})  and  (\ref{eq:estimation3}), we obtain for $0 < \delta < \delta_0$, 
  \begin{equation} \label{eq:funestimate}
  \int_{\Omega}(u-v) \sigma_{k +1}(\varphi)\leq   M_1 \frac{\kappa \delta^\alpha}{\delta^2} \int_{\Omega} (u-v) \sigma_k (\varphi)  + d(m,n) \, \kappa  \, R \delta^{\alpha}.
  \end{equation}
  
  To finish the proof of the last statement of the lemma, we argue by induction on $k$ for $0 \leq k \leq m$.
  When $k=0$,  the inequality is obviously satisfied with  $C_0 = 1$ and $\alpha_0 = 1$.
  
 Assume that  the  inequality holds for some integer $0 \leq k\leq m-1$ i.e. 
 
  \begin{equation}\label{eq:Hyprec}
  \int_{\Omega}( u-v) \sigma_k{(\varphi)}\leq C_{k}  \left[\Vert u-v \Vert_1\right]^{\alpha_k}.
  \end{equation}
  
  We will show that there exists $C_{k+1} > 0$ such that
  $$
  \int_{\Omega}(u-v) \sigma_{k+1}(\varphi)\leq C_{k+1}  \left[\Vert u-v \Vert_1\right]^{\alpha_{k+1}}.
  $$
   
   Indeed (\ref{eq:funestimate}) and (\ref{eq:Hyprec}) yields
$$  
 \int_{\Omega}(u-v) \sigma_{k+1}(\varphi)\leq
M_1 C_k \kappa \frac{\delta^{\alpha}}{\delta^2}  [\|u-v\|_1]^{\alpha_k} + d(m,n) \kappa  R \delta^{\alpha}.
$$ 

 We want to optimize the last estimate. 
 
Since  $ \|u-v\|_1 \leq 1$, we can take $\delta= \delta_0 [\|u-v\|_1]^{\alpha_k\slash 2} < \delta_0$ in the last inequality to obtain 
  
 \begin{eqnarray} \label{eq:lastestimate}
  \int_{\Omega}(u-v) \sigma_{k+1}(\varphi) &\leq & (M_1 C_k + d (m,n))\,  \kappa \left(\|u-v\|_1]^{\alpha_k \slash 2}\right)^{\alpha }  \nonumber \\
  & \leq & C_{k + 1}  R  [\|u-v\|_1]^{\alpha_{k+ 1}},
 \end{eqnarray}
 where $\alpha_{k + 1} := \alpha_k (\alpha \slash 2)$. 
 This proves the last statement of  the lemma.
 
 \smallskip
 \smallskip
 
   We now proceed to the proof of the first statement.
   As we saw before the main issue is to estimate uniformly the integrals like $ I'_k (v,\varphi)$,
 but  this is not the case in general. We  will rather consider the following integrals which behave much better :
  $$
  J_{k} (u,v,\varphi) := \int_{\Omega} (u-v)^m (dd^c \varphi)^{k} \wedge \beta^{n -k }
  $$
  By H\"older inequality, we have
  \begin{equation} \label{eq:Hineq2}
  \int_\Omega (u-v) (dd^c \varphi)^k \wedge \beta^{n-k} \leq \vert \Omega\vert^{ (m-1) \slash m}\left(J_k (u,v,\varphi)\right)^{1 \slash m},
  \end{equation}
 where $\vert \Omega\vert$ is the volume of $\Omega$.
  
  It is then enough to estimate $J_k (u,v,\varphi)$. We will proceed by induction on  $k$  ($0 \leq k \leq m$) to prove the following estimate
  \begin{equation} \label{eq:Jk}
  J_k (u,v) \leq C_k  R (1 + \Vert u - v\Vert_\infty^{m - 1})  \,  \|u-v\|_1^{\tilde{\alpha}_k}.
  \end{equation}
  
  If $k = 0$, we have $J_0 (u,v) \leq  \Vert u - v\Vert_{\infty}^{m - 1} \Vert u - v\Vert_1$. The inequality is satisfied with $\\tilde{\alpha}_0 = 1$ and $C'_0 = 1$
 
 Assume the estimate (\ref{eq:Jk}) is proved for some integer $0 \leq k \leq m-1$. To prove it for the integer $k + 1$,  we write as before
  $$
 J_{k+1} (u,v)  = A' (\delta) + B' (\delta), 
  $$
  where   
 $$
  A' (\delta):=\int_{\Omega}(u-v)^m dd^c \varphi_{\delta}\wedge(dd^c\varphi)^{k} \wedge\beta^{n-k},
 $$
     and 
 \begin{eqnarray*}
     B' (\delta) &:=& \int_{\Omega}(u-v)^m dd^c(\varphi - \varphi_{\delta} - \kappa \delta^\alpha)\wedge(dd^c\varphi)^{k} \wedge\beta^{n-k}.
 \end{eqnarray*}
 
The  first term $A' (\delta)$ is estimated as before
 \begin{equation} \label{eq:estimationA'}
   |A' (\delta)|\leq  M_1 \frac{\kappa \delta^{\alpha}}{\delta^2} J_k (u,v)
 \leq M_1 \, C_k \, \kappa  R  \,  \delta^{\alpha - 2}  \left[\Vert u-v \Vert_1\right]^{\alpha_k}.
 \end{equation}

We need to estimate the second term $B' (\delta)$. 

Since $ u - v = 0$ near the boundary, we can integrate by parts to get the following formula
\begin{equation} \label{eq:estimationB'0}
B' (\delta) = \int_{\Omega} (\varphi_{\delta} - \varphi +  \kappa \delta^\alpha) \left(- dd^c[ (u-v)^m]\right) \wedge (dd^c\varphi)^k\wedge\beta^{n-k-1}.
\end{equation}

 If $m = 1$, then $k = 0$ and 
 $$- dd^c [(u-v)   \wedge \beta^{n-1} \leq    dd^c v \wedge  \beta^{n-1}, 
 $$
  in the weak sense on $\Omega$.
 Since $0 \leq \varphi_{\delta} - \varphi +  \kappa \delta^\alpha \leq 2 \kappa \delta^\alpha$, it follows from (\ref{eq:estimationB'0}) that
\begin{equation} \label{eq:estimateB'1}
\vert B' (\delta) \vert  \leq  2 \kappa \delta^\alpha \int_\Omega dd^c v \wedge\beta^{n-1} \leq 2 R \, \kappa \, \delta^\alpha.
\end{equation}
 If $ m \geq 2$, a simple computation shows that
\begin{eqnarray*}
 - dd^c [(u-v)^m] & =&  -   m (u-v)^{m -1} dd^c (u - v)  \\
 &-& m (m-1) (u-v)^{m - 2} d(u-v) \wedge d^c(u-v) \\
 &\leq & - m (u-v)^{m -1} dd^c (u- v),
\end{eqnarray*}
Since $  dd^c u  \wedge (dd^c\varphi)^k\wedge\beta^{n-k-1} \geq 0$ weakly on $\Omega$, it follows that
$$
- dd^c [(u-v)^m] \wedge (dd^c\varphi)^k\wedge\beta^{n-k-1} \leq m (u-v)^{m -1} dd^c  v\wedge  (dd^c\varphi)^k\wedge\beta^{n-k-1},
$$
 weakly on $\Omega$.
Hence since $\varphi_{\delta} - \varphi +  \kappa \delta^\alpha \geq 0$,  that
$$
\vert B' (\delta)\vert \leq m  \int_{\Omega} (\varphi_{\delta}-\varphi +  \kappa \delta^\alpha)  (u-v)^{m -1} dd^c v  \wedge(dd^c\varphi)^k\wedge\beta^{n-k-1}.
$$
 Moreover, since $0 \leq \varphi_{\delta}-\varphi + \kappa \delta^\alpha  \leq 2 \kappa \delta^{\alpha}$ on $\Omega$, it follows from (\ref{eq:estimationB'0}) that
  
 \begin{equation} \label{eq:estimation2}
  |B' (\delta)| \leq   2 m \,  \kappa \,  \delta^{\alpha}  \int_\Omega (u-v)^{m -1} dd^c v \wedge (dd^c\varphi)^k\wedge\beta^{n-k-1}
 \end{equation}

Recall that $\beta = dd^c \psi_0$, where $\psi_0 (z) := \vert z\vert^2 -r_0^2$, where $r_0 > 0$ is choosen so that $\psi_0 \leq 0$ on $\Omega$. 
 Then the inequality (\ref{eq:estimation2}) implies that
 \begin{equation} \label{eq:estimationB'}
  |B' (\delta)| \leq   2 m  \kappa \delta^{\alpha}  \int_\Omega (u-v)^{m -1} dd^c v \wedge (dd^c\varphi)^k \wedge (dd^c \psi)^{m - k-1} \wedge\beta^{n-m}.
 \end{equation}
Since  $v, \varphi, \psi_0 \leq 0$ on $\Omega$, repeating the integration by parts $(m-1)$ times, we get 
\begin{equation} \label{eq:estimationC'}
  |B' (\delta)| \leq   2 m! \, \kappa \, \delta^{\alpha}  \Vert \varphi\Vert_{\infty}^k  \Vert \psi_0\Vert_{\infty}^{m - k-1} \int_\Omega (dd^c v)^m \wedge\beta^{n-m}.
  \end{equation} 
  Combining  the inequalities (\ref{eq:estimationA'}), (\ref{eq:estimationB'}) and   (\ref{eq:estimationC'}), we obtain for $0 < \delta < \delta_0$, 
  $$
  \int_{\Omega}(u-v)^m \sigma_{k +1}(\varphi)\leq M_1 \kappa \delta^{\alpha- 2} J_k (u,v)+ d'(m,n)   R \delta^{\alpha},
  $$
where $d'(m,n) = d'(m,n,\varphi,g) > 0$ is a uniform constant.
  Applying the induction hypothesis  (\ref{eq:Jk}), we get 
  
  $$
  \int_{\Omega}(u-v)^m \sigma_{k +1}(\varphi)\leq M_1 \kappa \delta^{\alpha- 2}  C_k \Vert u - v \Vert_{\infty}^{m -1}  \|u-v\|_1^{\alpha_k})+ d'(m,n)   R \delta^{\alpha},
  $$
  
  We want to optimize the last estimate. Since $ \|u-v\|_1 \leq 1$, we can take $\delta= \delta_0 [\|u-v\|_1]^{\alpha_k\slash 2} < \delta_0$ in the last inequality to obtain 
  
 \begin{eqnarray*}
  \int_{\Omega}(u-v)^m \sigma_{k+1}(\varphi) &\leq & (M_1 C_k \Vert u - v\Vert_{\infty}^{m -1} + R d (m,n))\,  \kappa \left(\|u-v\|_1]^{\alpha_k \slash 2}\right)^{\alpha } \\
  & \leq  & C_{k + 1} (1 + \Vert u - v\Vert_{\infty}^{m -1} )  R  [\|u-v\|_1]^{\alpha_{k+ 1}},
 \end{eqnarray*}
 where $\alpha_{k + 1} := \alpha_k (\alpha \slash 2)$ and and $C'_{k+1} := M_1 C'_k    + d'(m,n)$.  
 This proves the estimate (\ref{eq:Jk}) for $k+1$. Taking into account the inequality (\ref{eq:Hineq2}) we obtain the estimate of the lemma with appropriate constants. This finishes the proof of the second part of the lemma.  
\end{proof}
 
 \smallskip
 We also don't know if the lemma is true when the total mass of the Hessian measure $\sigma_m (\varphi) $ on $\Omega$  is infinite.

 \section{Proofs of the main results}

In this section we will give the proofs of Theorem A and Theorem B  stated in the introduction using the previous results.

\subsection{ Proof of Theorem A}

For the proof of Theorem A, we will use the same idea as \cite{KN19}. However, since our measure has not a compact support, we need to use the control on the behaviour of the mass of the $m$-Hessian  of the subsolution close to the boundary, given by Lemma \ref{lem:ComparisonIneq}.

 \begin{proof} 
We extend $\varphi$ as a H\"older continuous function on the whole of $\C^n$ with the same exponent and denote by $\varphi$ the extension (see (\ref{eq:Holderextension})).
 Then denote by $\varphi_{\delta}$ ($0 < \delta < \delta_0$) the smooth approximants of $\varphi $  defined by the formula (\ref{eq:regularization2}).
Then   $\varphi_\delta \in  \mathcal{SH}_m (\Omega_{\delta})\cap\mathcal{C}^{\infty}(\C^n)$.
 
 We consider the $m$-subharmonic envelope of   $\varphi_\delta$ on  $\Omega$  defined by the formula 
 $$
 \psi_\delta := \sup \{\psi \in \mathcal{SH}_m (\Omega) ; \psi \leq \varphi_\delta \, \, \, \hbox{on} \, \, \,  \Omega \}\cdot 
 $$ 

 It follows from Lemma \ref{lem:projection}  that $\psi_\delta \in \mathcal{SH}_m (\Omega)$ and $\psi_\delta \leq \varphi_\delta $ on $\Omega$. 
  
 Fix $0 < \delta < \delta_0$ and a compact set $K \subset \Omega_\delta$ and consider the set
  $$
 E :=\{3 \kappa \delta^\alpha u_K^*+ \psi_\delta<\varphi- 2 \kappa \delta^{\alpha}\} \subset \Omega.
 $$
 
 Since $\varphi $ is H\"older continuous on $\bar \Omega$, we have  $\varphi -  \kappa \delta^\alpha\leq \varphi_{\delta}  \leq \varphi + \kappa \delta^{\alpha}$ on $\Omega$ and then   $\varphi -  \kappa \delta^\alpha \leq \psi_\delta \leq  \varphi_{\delta}  \leq \varphi (z) + \kappa \delta^\alpha$  on $\Omega$. 
Therefore  $\liminf_{z \to \partial \Omega} (\psi_\delta - \varphi + \kappa \delta^\alpha ) \geq 0$, and then $E \Subset \Omega$. By the comparison principle, we conclude that
 
\begin{eqnarray} \label{eq:fundmentalestimate}
 \int_{E}(dd^c\varphi)^m\wedge\beta^{n-m} & \leq & \int_{E}(dd^c(3 \kappa \delta^\alpha u_K^* + \psi_{\delta}))^m\wedge\beta^{n-m} \nonumber \\
 & \leq & 3 \kappa L \delta^\alpha \int_{E}(dd^c(u_K^* + \psi_{\delta}))^m\wedge\beta^{n-m} \\
 &+&\int_{E}(dd^c\psi_{\delta})^m\wedge\beta^{n-m}, \nonumber
 \end{eqnarray}
 where $L := \max_{0 \leq j \leq m - 1} (3 \kappa \delta_0^\alpha)^j$.
 
 Observe that $-1 + \varphi -  \kappa \delta^\alpha \leq u_K^* + \psi_\delta  \leq \varphi + \kappa \delta^\alpha$ on $\Omega$, hence   $\vert u_K^* + \psi_\delta\vert \leq \sup_{\Omega}  \vert \varphi\vert + 1 + \kappa \, \delta_0^\alpha =: M_0$ on $\Omega$. 
 
 Therefore from inequality (\ref{eq:fundmentalestimate}), it follows that
 
\begin{equation} \label{eq:finalestimate1}
  \int_{E}(dd^c\varphi)^m\wedge\beta^{n-m} \leq 3 \kappa \delta^\alpha L  M_0^m \text{Cap}_m (E,\Omega) + \int_{E}(dd^c\psi_{\delta})^m\wedge\beta^{n-m}.
\end{equation}
 
  Since $\varphi$ is H\"older continuous on $\bar \Omega$, we have
 
 \begin{equation}\label{eq:1}
 dd^c\varphi_{\delta}\leq\frac{M_1 \kappa \delta^\alpha}{\delta^{2}}  \beta, \, \, \, \mathrm{on} \, \, \, \Omega,
 \end{equation}
 where $M_1 > 0$ is a uniform constant depending only on $\Omega$.

Hence by Theorem \ref{thm:obstacle}, we have
\begin{equation}\label{eq:1}
 (dd^c\psi_{\delta})^m\wedge\beta^{n-m}  \leq  (\sigma_m (\varphi_{\delta}))_+ \leq\frac{M_1^m \kappa^m \delta^{m \alpha}}{\delta^{2 m}}  \beta^n,
 \end{equation}
 in the sense of currents on $ \Omega$.
 
 Therefore
\begin{eqnarray*}
\int_{E}(dd^c\psi_{\delta})^m\wedge\beta^{n-m} &\leq &  M_1^{m} \kappa^m \delta^{ m (\alpha -2)} \lambda_{2 n} (E). 
%&\leq & B^m \delta^{m (\alpha - 2)} Vol (E).
 \end{eqnarray*}
 
From this estimate and the inequalities (\ref{eq:finalestimate1}) and (\ref{eq:1}), we deduce that
\begin{equation} \label{eq:finalestimate2}
\int_{E}(dd^c\varphi)^m\wedge\beta^{n-m} \leq  3 \kappa \delta^{\alpha} L M_0^m \text{Cap}_m (E,\Omega)+ M_1^{m} \kappa^m \delta^{(\alpha -2)m} \lambda_{2 n} (E).
\end{equation}

By the volume-capacity comparison inequality (\ref{eq:DK}), it follows that for any fixed $1 < r < \frac{m}{n - m}$, there exists a constant $N (r) > 0$ such that 
\begin{equation} \label{eq:volumeestimate}
\lambda_{2 n}(E) \leq N (r) [\text{Cap}_m(E,\Omega)]^{1 + r}.
\end{equation}

Since  $ E \subset\{u_K^* < - 1 \slash 3\}$,  by  the comparison principle we deduce the following inequality

 \begin{equation}\label{eq:2}
 \text{Cap}_m  (E,\Omega) \leq 3^{m}  \text{Cap}_m (K,\Omega).
 \end{equation}

Since $K \setminus \{u_K < u_K^*\} \subset E$ and $K \cap \{u_K < u_K^*\}$ has zero capacity, it follows that $\int_K (dd^c\varphi)^m\wedge\beta^{n-m} \leq  \int_E (dd^c\varphi)^m\wedge\beta^{n-m}$. 

Therefore if we set $c_m (\cdot) := \text{Cap}_m (\cdot,\Omega)$, we finally deduce from (\ref{eq:finalestimate2}), (\ref{eq:volumeestimate}) and  (\ref{eq:2}) that  for a fixed $0 < \delta < \delta_0$ and any compact set $K \subset \Omega_\delta$,  we have 
\begin{equation} 
\int_K (dd^c\varphi)^m\wedge\beta^{n-m} \leq C_0\kappa  \delta^{\alpha}  c_m (K) + C_1 \kappa^m \delta^{(\alpha -2) m }   [c_m(K)]^{1+ r}.
\end{equation}
where $C_0 :=  3^{m + 1}  L M_0^m$ and $C_1 :=  M_1^{m}  3^{m r}  N (r)$.

By inner regularity of the capacity, we deduce that the previous estimate holds for any Borel subset $B \subset \Omega_\delta$ i.e.
\begin{equation} \label{eq:estimate1}
\int_B (dd^c\varphi)^m\wedge\beta^{n-m} \leq C_0 \kappa \delta^{\alpha}  c_m (B) + C_1 \kappa^{m \alpha } \delta^{(\alpha -2) m}   [c_m(B)]^{1+r}.
\end{equation} 

Let $K \subset \Omega$ be any fixed compact set and $0 < \delta < \delta_0$. Then

$$
\int_K (dd^c\varphi)^m\wedge\beta^{n-m}  = \int_{K \cap \Omega_\delta} (dd^c\varphi)^m\wedge\beta^{n-m}  + \int_{K \setminus \Omega_\delta} (dd^c\varphi)^m\wedge\beta^{n-m}.
$$

We will estimate each term separately. By (\ref{eq:estimate1}) the first term is estimated easily:
$$
\int_{K \cap \Omega_\delta} (dd^c\varphi)^m\wedge\beta^{n-m}  \leq C_0 \kappa \delta^{\alpha}  c_m (K) + C_1 \kappa^{m \alpha}\delta^{-2m +  m \alpha}   [c_m(K)]^{1+ r}.
$$
To estimate the second term we apply Lemma \ref{lem:ComparisonIneq} for the Borel set $ B := K \setminus \Omega_\delta$. Since $\delta_B (\partial \Omega) \leq \delta$ we get
$$
\int_{K \setminus \Omega_\delta} (dd^c\varphi)^m\wedge\beta^{n-m}  \leq \kappa^m \delta^{m \alpha} c_m (K).
$$

Therefore we obtain the following estimate. For any $0 < \delta <\delta_0$ and any compact set $K \subset \Omega$, we have

 \begin{equation} \label{eq:fundIneq}
 \int_{K}(dd^c\varphi)^m\wedge\beta^{n-m} \leq C_0 \kappa \delta^{\alpha}  c_m (K) + C_1 \kappa^m \delta^{(\alpha -2) m }   [c_m(K)]^{1+ r} +  \kappa^m \delta^{m \alpha} c_m (K). 
\end{equation}  
% where  $C_2 := C_0 + C_1 > 0$ and $\tau := 1 + \frac{\alpha (r - 1)}{(2-\alpha) m + \alpha}$.
 
 We want to optimize the right hand side of (\ref{eq:fundIneq})  by taking $\delta := [c_m (K)]^{\frac{r}{(2-\alpha) m + \alpha}}$.
 
 Observe that if $ \delta_K(\partial \Omega) \leq [c_m (K)]^{\frac{r}{(2-\alpha) m + \alpha}}$, then by Lemma \ref{lem:ComparisonIneq} we get
 \begin{eqnarray} \label{eq:finalIneq1}
   \int_{K}(dd^c\varphi)^m\wedge\beta^{n-m}  \leq  \kappa^m [c_m (K)]^{1 + \frac{m \alpha r}{(2-\alpha) m + \alpha}}.
  \end{eqnarray}
 
  Now assume that $  [c_m (K)]^{\frac{r}{(2-\alpha) m + \alpha}} < \delta_K (\partial \Omega) \leq \delta_0$. Then we can take $\delta := [c_m (K)]^{\frac{r}{(2-\alpha) m + \alpha}}$ in inequality (\ref{eq:fundIneq}) and get
  
  \begin{equation}\label{eq:finalIneq2}
   \int_{K}(dd^c\varphi)^m\wedge\beta^{n-m} \leq (C_0 \kappa + C_1 \kappa^m  + \kappa^m) \, [c_m (K)]^{1 + \frac{\alpha r}{(2-\alpha) m + \alpha}}.
  \end{equation}
  Combining inequalities  (\ref{eq:finalIneq1} ) and (\ref{eq:finalIneq2}), we obtain the estimate of the theorem with the constant $A$ given by the following formula:
  \begin{equation} \label{eq:finalConst}
  A := C_0 \kappa + C_1 \kappa^m  + \kappa^m.
\end{equation}   
\end{proof}

\subsection{Proof of Theorem B}

Now we are ready to prove Theorem B from the introduction using Theorem A and Lemma \ref{lem:ModC}.

\begin{proof}  According to Theorem \ref{thm:boundedsubsolution}, we know that there is a unique function $u\in \mathcal{SH}_m(\Omega)\cap L^{\infty}(\overline{\Omega}) $ such that 
$$ 
(dd^c u)^m\wedge\beta^{n-m}=\mu,
$$
in the weak sense on $\Omega$ and $ u=g $ on $\partial\Omega$.

We want to prove that $u$ is H\"older continuous up to the boundary.

For  $0 < \delta < \delta_0$ and  denote as before by 
 $u_{\delta}(z)$ the $\delta$-regularization of $u$. %$\Omega_{\delta}:=\{z\in\Omega:d(z,\partial\Omega)>\delta\}$$
Recall that  $u_\delta$ is smooth on $\C^n$ and $m$-subharmonic on $\Omega_\delta$.

The first step is to use the H\"older continuity of the boundary datum $g$ to construct global barriers to show that $u$ is $m$-subharmonic near the boundary and to deduce bounded $m$-subharmonic global approximants $\tilde{u}_{\delta}$  close to ${u}_{\delta}$ on $\Omega_\delta$.

By \cite{Ch16} there exists a continuous maximal $m$-subharmonic function $w \in \mathcal{SH} (\Omega) \cap C^{\alpha} (\bar \Omega)$ such that $w = g$ on $\partial \Omega$. Then $v : = w + \varphi \in  \mathcal{SH} (\Omega) \cap \mathcal{C}^{\alpha} (\bar \Omega)$ is a subsolution to the Dirichlet problem (\ref{eq:DirPb}) such that $v = g$ on $\partial \Omega$. Hence $v \leq u \leq w$. This proves that $u$ is H\"older continuous near the boundary $\partial \Omega$ (see Remark  \ref{rem:HolderBoundary}). 

Therefore to prove that $u$ is H\"older continuous on $\bar{\Omega}$ with exponent $\theta \in ]0,1[$, it's enough by Lemma \ref{lem:sup-mean} to  prove that there exists a constant $L > 0$ such that for $\delta < \delta_0$, 
\begin{equation} \label{eq:Holdercontinous} 
\sup_{\Omega_\delta} (u_\delta  - u) \leq C \delta^{\theta}.
\end{equation}

We claim that there exists a constant $\kappa > 0$ such that for $z\in\partial\Omega_{\delta}$, we have $ u (z) \geq {u}_{\delta}(z)  -  \kappa \delta^{\alpha} $. Indeed fix $z\in\partial\Omega_{\delta}$. Then there exists $\zeta \in \partial \Omega$ such that $\vert z -\zeta \vert  = \delta$. Since $v \leq u \leq w$ on $\Omega$ and they are equal on $\partial \Omega$, it follows that 
\begin{eqnarray*}
u_{\delta} (z) &  \leq &  w_\delta (z)  \leq  w (z) + \kappa_w \delta^\alpha \\
& \leq &  w (\zeta) + 2 \kappa_w \delta^\alpha = v (\zeta) +  2 \kappa_w \delta^\alpha \\
& \leq & v (z) +( \kappa_v + 2 \kappa_w)  \delta^\alpha \\
& \leq & u (z) + \kappa \delta^\alpha,
\end{eqnarray*}
where $\kappa := \kappa_v + 2 \kappa_w$ and $ \kappa_v$ (resp. $\kappa_w$) is the H\"older constant of $v$ (resp. $w$). This proves our claim. 

Therefore the following function
$$
 \tilde{u}_{\delta}:= \left\{ \begin{array}{lcl}
\max\{{u}_{\delta} -  \kappa \delta^{\alpha},u \} &\hbox{on} & \Omega_{\delta}, \\
u  &\hbox{on} & \Omega\setminus\Omega_{\delta}
\end{array}\right.
$$
is $m$-subharmonic and bounded on $\Omega$ and satisfies 
$0 \leq \tilde{u}_{\delta}(z) - u (z) =(u_\delta  (z) - u (z) - \kappa \delta^{\alpha})_+ \leq  u_\delta  (z) - u (z)$ for $z\in \Omega_{\delta}$ and $\tilde{u}_{\delta}(z) - u (z) = 0$ on $\Omega \setminus \Omega_\delta$.

Moreover, since $ \tilde{u}_{\delta} \geq u$ on $ \Omega$ and $\tilde{u}_{\delta} = u$ on $\Omega \setminus \Omega_\delta$, Corollary \ref{coro:Comparison Principle} implies that for any $0 < \delta < \delta_0$, we have

$$
\int_{\Omega}(dd^c\tilde{u}_{\delta})^m\wedge\beta^{n-m}\leq\int_{\Omega}(dd^c u)^m\wedge\beta^{n-m} \leq \mu (\Omega) < + \infty.
$$

The second step is to apply stability estimates. Since $\tilde u_\delta = u$ on $\Omega \setminus \Omega_\delta$,  Proposition \ref{prop:stability} implies that for 
$$
0 <\gamma <  \gamma (m,n,\alpha):= 
 \frac{m \alpha}{ m (m + 1) \alpha  +  (n-m) [(2 - \alpha)m + \alpha]},
 $$
there exists a constant $D_\gamma > 0$ such that  any $0 < \delta < \delta_0$, 

\begin{equation} \label{eq2}
\sup_{\Omega}(\tilde{u}_{\delta}-u) \leq  D_{\gamma} \left(\int_{\Omega}(\tilde{u}_{\delta}-u) d\mu\right)^{\gamma}.
\end{equation}

  On the other hand, since $\mu\leq (dd^c\varphi)^m\wedge\beta^{n-m}$ on $\Omega$,   it follows from Theorem A that we can apply Lemma \ref{lem:ModC},  if we can ensure that $\Vert (\tilde{u}_{\delta}-u \Vert_{1} :=   \int_{\Omega}(\tilde{u}_{\delta}-u) d\lambda_{2n} \leq 1$ for $\delta > 0$ small enough. 
 
 Indeed by the estimate (\ref{eq:PJ2}), we see that there exists a uniform constant $b_n > 0$ such that for $0 < \delta < \delta_0$,
 \begin{equation*}
 \int_{\Omega_\delta}({u}_{\delta}-u) d \lambda_{2 n} \leq  b_n \, \text{osc}_{\Omega} u \leq 1,
\end{equation*}
for $0 < \delta <  \delta_1$, where $\delta_1 > 0$ is small enough.

 To prove (\ref{eq:Holdercontinous}), we will consider the two cases separately.
  
 \smallskip
  
  1)   Assume first  that $g \in \mathcal C^{1,1} (\partial \Omega)$.  
  By the inequality (\ref{eq:PJ1}) we have for $0 < \delta < \delta_0$,
  \begin{equation} \label{eq:PJ}
 \int_{\Omega_\delta}({u}_{\delta}-u) d \lambda_{2 n} \leq  b_n \delta^2\Vert \Delta u\Vert_{\Omega_\delta}.
\end{equation}
  By the inequality (\ref{eq:claim})  of  Lemma \ref{lem:Cegrell2},  we have for $0 < \delta < \delta_0$,

 \begin{eqnarray} \label{eq:PJ0}
   \|\Delta u\|_{\Omega_\delta} \leq  \int_\Omega dd^c u \wedge \beta^{n - 1} 
  &\leq  & d(m,n) \left(M' + \mu (\Omega)^{1 \slash m}\right) .
 \end{eqnarray}
  
  On the other hand, applying the inequality (\ref{eq:MocEst1}) of  Lemma \ref{lem:ModC}, we get for $0 < \delta < \delta_1$, 

 \begin{eqnarray} \label{eq:stab}
\int_{\Omega}(\tilde{u}_{\delta}-u)d\mu & \leq & C_m  \left(\int_{\Omega}(\tilde{u}_{\delta}-u)(z)d\lambda_{2 n} (z)\right)^{\alpha_m} \\
 &\leq  & C_m\left(\int_{\Omega_{\delta}}({u}_{\delta}(z)-u(z)d\lambda_{2 n} (z)\right)^{\alpha_m}, \nonumber 
 \end{eqnarray}
 where the last inequality follows from the fact that $0 \leq \tilde{u}_{\delta}-u \leq \tilde{u}_{\delta}-u$ on $\Omega$ and $\tilde{u}_{\delta}= u$ on $\Omega \setminus \Omega_\delta$.
 
Therefore taking into account the inequalities (\ref{eq:PJ}), (\ref{eq:PJ0}) and (\ref{eq:stab}),  we get  for $0 < \delta < \delta_1$, 

 \begin{equation} \label{eq:estimatemu}
\int_{\Omega}(\tilde{u}_{\delta}-u)d\mu  \leq  C_m' \delta^{2 \alpha_m}.
\end{equation} 
 
 Finally  from  (\ref{eq2}) and (\ref{eq:estimatemu}) we conclude that  for $0 < \delta < \delta_1$, 
 \begin{equation} \label{eq:finalestimate}
 \begin{array}{lcl}
 \sup_{\Omega_{\delta}}( {u}_{\delta}-u)&\leq&\sup_{\Omega}(\tilde{u}_{\delta}-u)+ \kappa \delta^{\alpha} \\
 &\leq &  {C''_m} \delta^{2\gamma\alpha_m} + \kappa \delta^{\alpha}.
 \end{array}
 \end{equation}
This proves the required estimate (\ref{eq:Holdercontinous}) with $\theta = 2 \gamma \alpha_m < \alpha$.

 \smallskip
 \smallskip
 
2) In the general case the estimate (\ref{eq:PJ0}) on   $\|\Delta u\|_{\Omega_\delta}$ is not valid anymore.
 
  However  by (\ref{eq:PJ}) and (\ref{eq:PJ2}), we get for $0 < \delta < \delta_1$,
   \begin{equation} \label{eq:6}
\int_{\Omega_\delta}({u}_{\delta}-u)d\lambda_{2n}  \leq  {\tilde C}'_m \delta \leq 1.
\end{equation}

Therefore  taking into account    (\ref{eq2}),  (\ref{eq:6}) and the inequality (\ref{eq:MocEst2}) of  Lemma \ref{lem:ModC},  we get   for $0 < \delta < \delta_1$, 
 $$
  \sup_{\Omega_{\delta}} ({u}_{\delta}-u)  \leq    \tilde C_m \, \delta^{\gamma\alpha_m \slash m} + \kappa \delta^{\alpha},
$$
 since $\Vert \tilde u_{\delta} - u\Vert_{\infty} \leq \text{osc}_{\Omega} u$.
 This  proves the required estimate (\ref{eq:Holdercontinous}) with $\theta = \gamma\alpha_m \slash m < \alpha$.  
 \end{proof}
\smallskip
\smallskip
\begin{remark} If  we assume that $g \in C^{\theta} (\partial \Omega)$, it is easy to see that the solution  in Theorem B is H\"older continuous with any exponent $0 < \alpha' < \min \{\theta \slash 2, \gamma \tilde \alpha_m\}$. 
\end{remark}

\smallskip

\noindent{\bf Warning :}  This paper presents a corrected version of the one published recently in  \cite{BZ20}. Indeed the proof of \cite[Theorem B]{BZ20} was not complete because of an error in \cite[Lemma 4.2]{BZ20}. An erratum has been submitted to the journal and hopefully it will appear soon.
 
\smallskip
\smallskip

\noindent{\bf Aknowledgements :} 
   The authors  are indebted to Hoang Chinh Lu for his very careful reading of the first version of this paper and for valuable comments that helped to correct some errors and to improve the presentation of the paper. They also would like to thank Vincent Guedj for interesting discussions and  Ngoc Cuong Nguyen for useful exchanges about his earlier work on this subject.                                                                                                                             

 This project started when the first author was visiting the Institute of Mathematics of Toulouse (IMT) during the spring $2017$ and $2018$. She would like to thank IMT for the invitation and for providing excellent research conditions. 

%This paper is a revised version of the  previous one which is published in  Journal de l'\'Ecole Polytechnique, Tome 7 (2020), 981-1007 and takes into account of the Erratum published in   . The authors thank the referees %for their useful comments.

\end{document}